\documentclass{article}
\usepackage[utf8]{inputenc}
\usepackage{amsmath}
\usepackage{amsfonts}
\usepackage{amsthm}
\usepackage{amssymb}
\usepackage{mathtools}
\usepackage{dsfont}
\usepackage{geometry}
\usepackage{hyperref}
\usepackage{mathrsfs}
\usepackage{float}
\usepackage{subcaption}
\usepackage{aligned-overset}
\usepackage{authblk}
\usepackage{siunitx}
\usepackage{booktabs}

\newcommand{\lp}{(}\newcommand{\rp}{)}

\newcommand{\NN}{\mathbb{N}}
\newcommand{\RR}{\mathbb{R}}
\newcommand{\Sp}{\mathbb{S}}
\newcommand{\PP}{\mathbb{P}}
\newcommand{\EE}{\mathbb{E}}
\newcommand{\blank}{{\mspace{2mu}\cdot\mspace{2mu}}}
\DeclareMathOperator*{\argmin}{arg\,min}
\allowdisplaybreaks

\theoremstyle{plain}
\newtheorem{theorem}{Theorem}
\newtheorem{lemma}{Lemma}
\newtheorem{corollary}{Corollary}
\theoremstyle{remark}
\newtheorem{remark}{Remark}
\newtheorem{example}{Example}
\theoremstyle{definition}
\newtheorem{definition}[theorem]{Definition} 

\title{Nonparametric inference for Poisson-Laguerre tessellations}
\author{Thomas van der Jagt}
\author{Geurt Jongbloed}
\author{Martina Vittorietti}
\affil{Delft Institute of Applied Mathematics, Delft University of Technology.}
\date{January 15, 2025}

\begin{document}

\maketitle

\begin{abstract}
    In this paper, we consider statistical inference for Poisson-Laguerre tessellations in $\RR^d$. The object of interest is a distribution function $F$ which uniquely determines the intensity measure of the underlying Poisson process. Two nonparametric estimators for $F$ are introduced which depend only on the points of the Poisson process which generate non-empty cells and the actual cells corresponding to these points. The proposed estimators are proven to be strongly consistent, as the observation window expands unboundedly to the whole space. We also consider a stereological setting, where one is interested in estimating the distribution function associated with the Poisson process of a higher dimensional Poisson-Laguerre tessellation, given that a corresponding sectional Poisson-Laguerre tessellation is observed.
\end{abstract}

\section{Introduction}
Tessellations have proven to be useful in a wide range of fields. For example, a Poisson-Voronoi tessellation may serve as a model for a wireless network \cite{Baccelli2009}. In cosmology, Voronoi tessellations can be used to describe the distribution of galaxies \cite{vdWeygaert1994}. Another important field of application is materials science. There, a Laguerre tessellation may be fitted to the so-called microstructure of a material. For instance, Laguerre tessellations were found to be accurate models for foams \cite{Lautensack2008b}, \cite{Liebscher2015}, sintered alumina \cite{Falco2017} and composites \cite{Wu2010}. A challenge in this field is that in practice often only 2D microscopic images of cross sections of the 3D microstructure can be obtained. By studying a 3D object via a 2D slice there is evidently a loss of information. Inverse problems of this type, which involve the estimation of higher dimensional information from lower dimensional observations, belong to the field of stereology.

In this paper, we focus on statistical inference for a particular class of random tessellations known as Poisson-Laguerre tessellations. We do this both for the case where one directly observes a tessellation as well as for the case where the observed tessellation is obtained by intersecting a higher dimensional tessellation with a hyperplane. The latter type of tessellation is often referred to as a sectional tessellation. A Laguerre tessellation in $\RR^d$ is defined via a set of weighted points $\eta = \{(x_1,h_1),(x_2,h_2),\dots\}$, called generators. Here, $x_i$ is a point in $\RR^d$ and $h_i>0$ its weight. Each generator corresponds to a set, which is either a polytope or the empty set. This set is usually called a cell and we may also say that a generator generates this cell. The non-empty cells form a tessellation, meaning that these cells have disjoint interiors and the union of these cells equals $\RR^d$. We refer to the subset $\eta^* \subset \eta$ of points which generate non-empty cells as the extreme points of $\eta$. A Poisson-Laguerre tessellation, which is a random tessellation, is obtained by taking $\eta$ to be a Poisson (point) process on $\RR^d \times (0,\infty)$. The intensity measure of $\eta$ is assumed to be of the form $\nu_d \times \mathbb{F}$. Here, $\nu_d$ is Lebesgue measure on $\RR^d$ and $\mathbb{F}$ is a non-zero locally finite measure concentrated on $(0,\infty)$. An example of a realization of a Poisson-Laguerre tessellation, and the corresponding realization of extreme points is shown in Figure \ref{figure_discrete_laguerre_example}. Random Laguerre tessellations generated by an independently marked Poisson process were first studied in \cite{Lautensack2007} and \cite{Lautensack2008}. We mostly follow the description of Poisson-Laguerre tessellations as given in \cite{Gusakova2024_2}. Additionally, we will also rely on the result from \cite{Gusakova2024_2} which states that the sectional Poisson-Laguerre tessellation is again a Poisson-Laguerre tessellation. The so-called $\beta$-Voronoi tessellation as introduced in \cite{Gusakova2022} may be seen as a parametric model for a Poisson-Laguerre tessellation. In \cite{Gusakova2024} it was shown that the sectional Poisson-Voronoi tessellation is in fact a $\beta$-Voronoi tessellation.

Because tessellations are usually not directly observed in nature, typically the first step towards statistical inference for tessellations is a reconstruction step. Such a reconstruction method is used to obtain a tessellation from an image, for details see section 9.10.1 in \cite{Chiu2013} and references therein. Therefore, when applying the methodology in this paper to real data, it needs to be combined with such a reconstruction method. It is important to point out that the reconstruction methods used in \cite{Lautensack2008b}, \cite{Liebscher2015} and \cite{Seitl2021} reconstruct a Laguerre tessellation along with the extreme points simultaneously. Effectively, statistical inference for a Poisson-Laguerre tessellation is then reduced to statistical inference for the point process $\eta^*$. This appears to be the most common approach towards statistical inference for random Laguerre tessellations, and this is also the approach we take. For instance, in \cite{Seitl2022} a methodology is proposed for statistical inference for Laguerre tessellations, where parametric models are considered for the underlying point process. In \cite{Stoyan2021}, a Laguerre tessellation, along with the corresponding extreme points, is fitted to real data. Furthermore, a statistical analysis is performed on this point process of extreme points. 

\begin{figure}[t!]
    \centering
    \makebox[\textwidth]{\makebox[\textwidth]{
    \begin{subfigure}[t]{0.5\textwidth}
        \centering
        \includegraphics[width=0.9\linewidth]{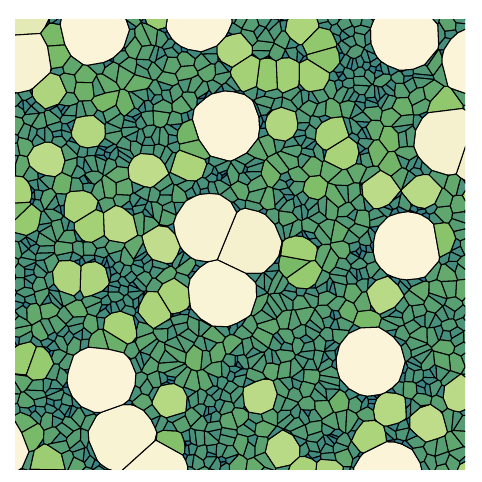}
    \end{subfigure}
    \begin{subfigure}[t]{0.5\textwidth}
        \centering
        \includegraphics[width=0.9\linewidth]{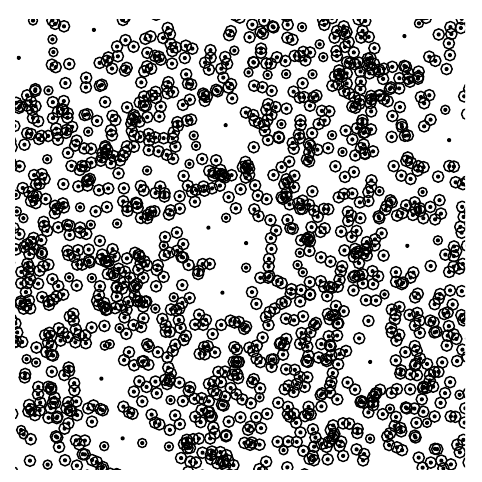}
    \end{subfigure}\hfill}}
    \caption{Left: A realization of a planar Poisson-Laguerre tessellation. Cells are colored according to their area. Right: The corresponding realization of extreme points. Around each point there is a circle with radius proportional to the weight of the point.}
    \label{figure_discrete_laguerre_example}
\end{figure}

Recall that the intensity measure of the underlying Poisson process $\eta$ is assumed to be of the form $\nu_d \times \mathbb{F}$. For $z \geq 0$ we define $F(z) := \mathbb{F}((0,z])$, the distribution function of $\mathbb{F}$. Note that this distribution function is the only parameter in this model to be estimated. In this paper, we define nonparametric estimators for $F$. These estimators for $F$ depend on both the observed Laguerre cells in a bounded observation window as well as the points of $\eta^*$ in the same window. These estimators are proven to be consistent as the observation window expands unboundedly to the whole of $\RR^d$. Additionally, we consider the stereological setting where the observed Poisson-Laguerre tessellation in $\RR^{d-1}$ is obtained by intersecting a Poisson-Laguerre tessellation in $\RR^d$ with a hyperplane. Based on this observed sectional tessellation we introduce an estimator for the distribution function corresponding to the Poisson process of the higher dimensional tessellation.

This paper is organized as follows. In section \ref{section_preliminaries} we introduce necessary notation and definitions. Then, the main mathematical object of interest, the Poisson Laguerre tessellation, is discussed in section \ref{section_introduction_pois_laguerre}. In section \ref{section_first_estimator} we introduce our first estimator for $F$, which is based on a thinning of the extreme points. A second estimator for $F$ is introduced in section \ref{section_second_estimator}, which depends on all observed extreme points, as well as the volumes of the corresponding Laguerre cells. In section \ref{section_stereology} we consider statistical inference for Poisson-Laguerre tessellations in a stereological setting. Finally, in section \ref{section_simulations} we perform a simulation study for the proposed estimators

\section{Preliminaries}\label{section_preliminaries}

In this section we introduce notation and various definitions which we need throughout this paper. Let $\nu_d$ denote Lebesgue measure on $\RR^d$, and $\sigma_{d-1}$ Lebesgue measure on the sphere $\Sp^{d-1} = \{x \in \RR^d: \Vert x \Vert = 1\}$, also known as the spherical measure. Given $x \in \RR^d$ and $r > 0$, we write $B(x, r) = \{y \in \RR^d : \Vert x-y\Vert < r\}$ and $\bar{B}(x, r) = \{y \in \RR^d : \Vert x-y\Vert \leq r\}$ for the open and closed ball respectively, with radius
$r$ centered at $x$. We also introduce the following constants:
\[\kappa_d := \nu_d\left(\Bar{B}(0,1)\right) = \frac{2\pi^\frac{d}{2}}{\Gamma\left(1 + \frac{d}{2} \right)} \text{, \ and \ } \omega_d := \sigma_{d-1}\left(\Sp^{d-1}\right) = \frac{2\pi^\frac{d}{2}}{\Gamma\left(\frac{d}{2} \right)}\]
Let $A, B \subset \RR^d$, then the sum of sets is defined as: $A + B = \{a + b: a \in A, b \in B\}$. If $x \in \RR^d$, we also write: $A + x = \{a+x: a \in A\}$. Let $\mathcal{F}_{+}$ denote the space of all (not necessarily bounded) distribution functions on $(0,\infty)$.

We now introduce several definitions related to point processes. While these definitions are valid for point processes in much more general spaces, in this paper we only consider point processes on $\RR^d \times (0,\infty)$. For more background on the theory of point processes we refer to \cite{Daley2008} and \cite{Last2018}. Suppose $\mathbb{X} = \RR^d \times (0,\infty)$, and let $(\Omega, \mathcal{A},\PP)$ be a probability space. A measure $\mu$ on $\mathbb{X}$ is locally finite if $\mu(B) < \infty$ for all bounded $B \in \mathcal{B}(\mathbb{X})$. Here, $\mathcal{B}(\mathbb{X})$ denotes the Borel $\sigma$-algebra of $\mathbb{X}$. Let $\mathbf{N}(\mathbb{X})$ denote the space of locally finite counting measures (integer-valued measures) on $\mathbb{X}$. We equip $\mathbf{N}(\mathbb{X})$ with the usual $\sigma$-algebra $\mathcal{N}(\mathbb{X})$, which is the smallest $\sigma$-algebra on $\mathbf{N}(\mathbb{X})$ such that the mappings $\mu \mapsto \mu(B)$ are measurable for all $B \in \mathcal{B}(\mathbb{X})$. A point process on $\mathbb{X}$ is a random element $\eta$ of $(\mathbf{N}(\mathbb{X}), \mathcal{N}(\mathbb{X}))$, that is a measurable mapping $\eta:\Omega \to \mathbf{N}(\mathbb{X})$. The intensity measure of a point process $\eta$ on $\mathbb{X}$ is the measure $\Lambda$ defined by $\Lambda(B):= \EE(\eta(B))$, $B \in \mathcal{B}(\mathbb{X})$. 

\begin{definition}
    Suppose $\Lambda$ is a $\sigma$-finite measure on $\mathbb{X}$. A Poisson process with intensity measure $\Lambda$ is a point process $\eta$ on $\mathbb{X}$ with the following two properties:
\begin{enumerate}
    \item For every $B \in \mathcal{B}(\mathbb{X})$, the random variable $\eta(B)$ is Poisson distributed with mean $\Lambda(B)$.
    \item For every $m \in \NN$ and pairwise disjoint sets $B_1,\dots,B_m \in \mathcal{B}(\mathbb{X})$, the random variables\\ $\eta(B_1),\dots,\eta(B_m)$ are independent.
\end{enumerate}
\end{definition}

Let $\delta$ denote the Dirac measure, hence for $x \in \mathbb{X}$ and $B \in \mathcal{B}(\mathbb{X})$: $\delta_x(B) = \mathds{1}\{x \in B\}$. A counting measure $\mu$ on $\mathbb{X}$ is called simple if $\mu(\{x\}) \leq 1$ for all $x \in \mathbb{X}$. As such, a simple counting measure has no multiplicities. Similarly, a point process $\eta$ on $\mathbb{X}$ is called simple if $\PP\left(\eta(\{x\}) \leq 1, \text{ } \forall x \in \mathbb{X} \right) = 1$. Let $\mathbf{N}_s(\mathbb{X})$ be the subset of $\mathbf{N}(\mathbb{X})$ containing all simple measures. Define: $\mathcal{N}_s(\mathbb{X}):= \{A \cap \mathbf{N}_s(\mathbb{X}): A \in \mathcal{N}(\mathbb{X})\}$. Then, a simple point process on $\mathbb{X}$ may be seen as a random element $\eta$ of $(\mathbf{N}_s(\mathbb{X}), \mathcal{N}_s(\mathbb{X}))$. If a point process is simple it is common to identify the point process with its support, and view the point process as a random set of discrete points in $\mathbb{X}$. We may for example write $x \in \eta$ instead of $x \in \mathrm{supp}(\eta)$. It is common practice to switch between the interpretations of a simple point process as a random counting measure or as a random set of points, depending on whichever interpretation is more convenient. We will also do this throughout this paper. Enumerating the points of a simple point process in a measurable way we may write:
\[\eta = \{x_1,x_2,\dots\}, \text{ \ and \ } \eta = \sum_{i \in \NN}\delta_{x_i}.\]
For $v \in \RR^d$ let $S_v$ denote the shift operator. Suppose $\eta = \{(x_1,h_1), (x_2, h_2),\dots\}$ is a point process with $x_i \in \RR^d$ and $h_i > 0$. Then, we define $S_v\eta := \{(x_1-v,h_1),(x_2-v,h_2),\dots\}$. Additionally, for a deterministic set $B \subset \RR^d \times (0,\infty)$ we define $S_vB := \{(x+v,h): (x,h) \in B\}$. Note that in the random counting measure interpretation of a point process, the definition is as follows: $S_v\eta(B) := \eta\left(S_vB \right)$, for $B\in \mathcal{B}(\RR^d\times(0,\infty))$. This is indeed consistent with the previous definition since $S_v\eta(B) = \sum_i \delta_{(x_i,h_i)}(S_vB) = \sum_i \delta_{(x_i-v,h_i)}(B)$. We call $\eta$ stationary if $S_v\eta$ and $\eta$ are equal in distribution for all $v \in \RR^d$. Throughout this paper, $(W_n)_{n \geq 1}$ is a fixed convex averaging sequence. That is, each $W_n \subset \RR^d$ is convex and compact, and the sequence is increasing: $W_n \subset W_{n+1}$. Finally, the sequence $(W_n)_{n \geq 1}$ expands unboundedly: $\sup\{r \geq 0: B(x,r) \subset W_n \text{ for some } x \in W_n\} \to \infty$ as $n \to \infty$.

\section{Poisson-Laguerre tessellations}\label{section_introduction_pois_laguerre}
In this section we describe the main mathematical object of interest in this paper, the Poisson-Laguerre tessellation. This random tessellation is a generalization of the well-known Poisson-Voronoi tessellation, and was first studied in \cite{Lautensack2007} and \cite{Lautensack2008}. We will mostly follow the description of the Poisson-Laguerre tessellation as given in \cite{Gusakova2024_2}, which is subtly different. Let us start with the definition of a tessellation:
\begin{definition}
    A tessellation of $\RR^d$ is a countable collection $T = \{C_i: i\in \NN\}$, of sets $C_i \subset \RR^d$ (the cells of the tessellation) such that:
    \begin{itemize}
        \item $\mathrm{int}(C_i)\cap \mathrm{int}(C_j) = \emptyset$, if $i \neq j$. 
        \item $\cup_{i \in \NN} C_i = \RR^d$.
        \item $T$ is locally finite: $\#\{i\in \NN: C_i \cap B \neq \emptyset\} < \infty$ for all bounded $B \in \mathcal{B}(\RR^d)$.
        \item Each $C_i$ is a compact and convex set with interior points.
    \end{itemize}
\end{definition}

Now, we will introduce the Laguerre diagram. Let $\varphi = \{(x_i,h_i)\}_{i \in \NN}$, with $x_i \in \RR^d$ and $h_i > 0$. Assume moreover that $x_i \neq x_j$ for $i\neq j$. The Laguerre cell associated with $(x,h) \in \varphi$ is defined as:
\begin{equation}
    C((x,h), \varphi) = \left\{y \in \RR^d: \Vert y-x \Vert^2 + h \leq \Vert y-x' \Vert^2 + h' \text{ for all } (x',h') \in \varphi\right\}. \label{eq_Laguerre_cell_def}
\end{equation}
The Laguerre diagram generated by $\varphi$ is the set of non-empty Laguerre cells, and is denoted by $L(\varphi)$: 
\[L(\varphi) := \left\{C((x,h), \varphi): (x,h) \in \varphi \text{ and } C((x,h), \varphi) \neq \emptyset\right\}.\]
A Laguerre diagram is not necessarily a tessellation, conditions on $\varphi$ are needed to ensure that $L(\varphi)$ is locally finite and that all cells are bounded. As we will discuss in a moment, the random Laguerre diagrams we consider are in fact tessellations. A Laguerre diagram has an interesting interpretation as a crystallization process. From the definition of a Laguerre cell it follows that:
\[x \in C((x_i,h_i),\varphi) \iff \exists t \geq h_i: x \in \bar{B}\left(x_i,\sqrt{t - h_i}\right) \text{ and } x \notin \bigcup_{j \neq i} B\left(x_j, \sqrt{(t-h_j)_{+}}\right),\]
with $(x)_{+} = \max\{x, 0\}$. Hence, we may consider the ball $B_i(t) := \bar{B}(x_i, \sqrt{(t-h_i)_{+}})$ which starts growing at time $t = h_i$. The ball initially grows fast, and then its growth slows down. If $B_i$ is the first ball to hit a given point $x \in \RR^d$, then $x \in C((x_i,h_i),\varphi)$. It is possible that $x_i$ lies in another cell $C((x_j,h_j),\varphi)$, $i\neq j$ and yet $C((x_i,h_i),\varphi)$ may be non-empty. It is also possible that a pair $(x_i,h_i)$ does not generate a cell, essentially because its ball starts growing too late. A visualization of the crystallization process is given in Figure \ref{fig:crystallization}.

\begin{figure}[b!]
    \centering
    \makebox[\textwidth]{\makebox[\textwidth]{
    \begin{subfigure}[t]{0.25\textwidth}
        \centering
        \includegraphics[width=\linewidth]{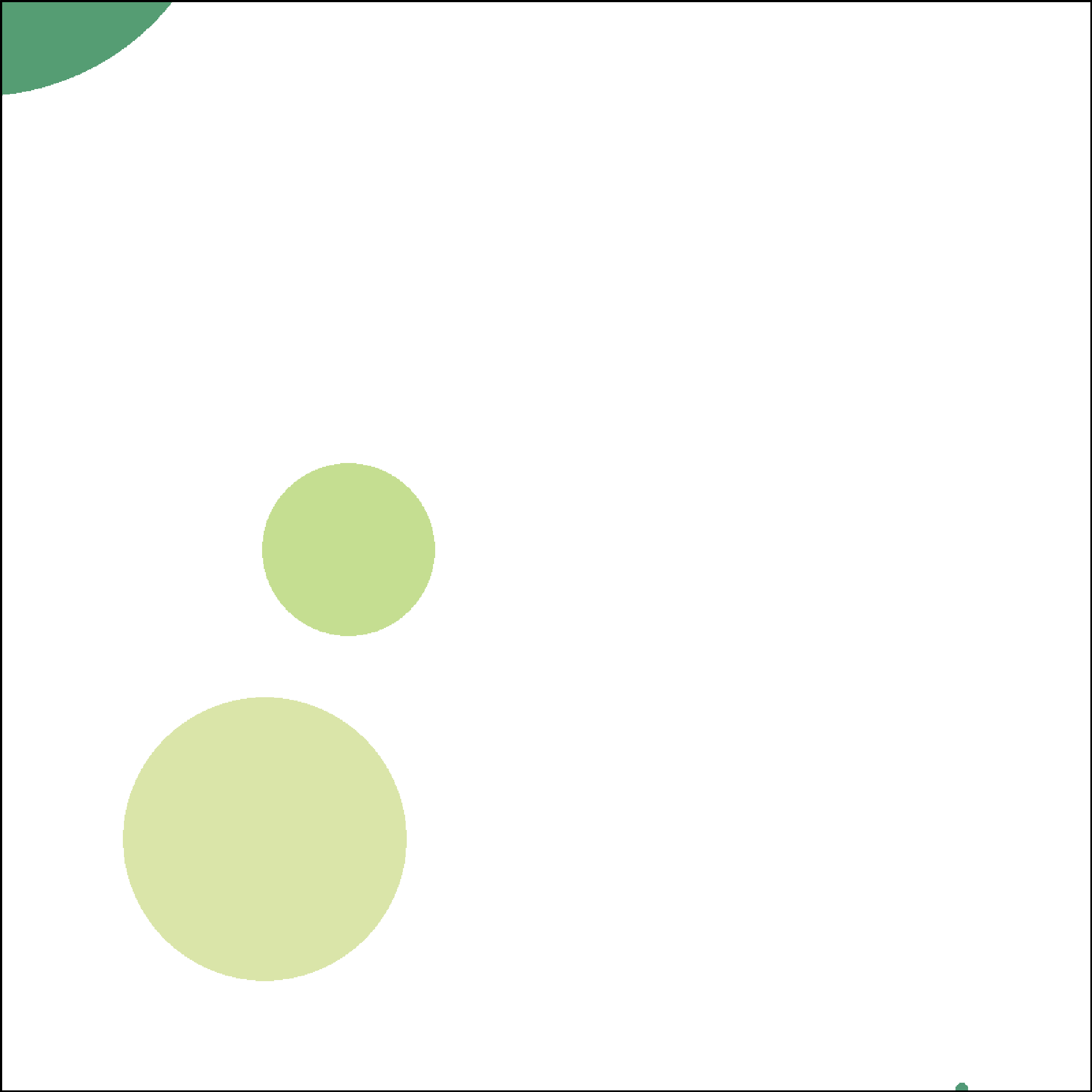}
    \end{subfigure}
    \begin{subfigure}[t]{0.25\textwidth}
        \centering
        \includegraphics[width=\linewidth]{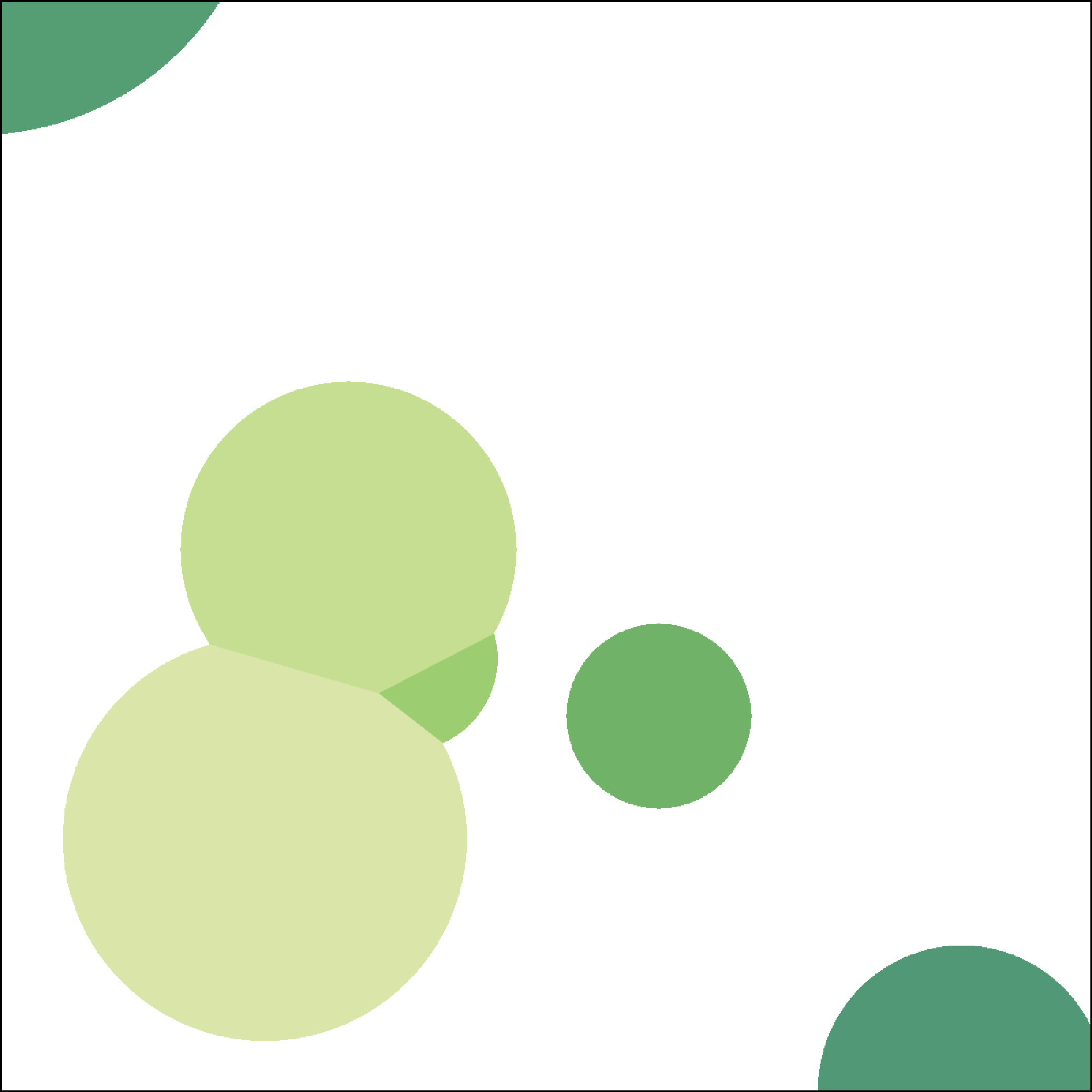}
    \end{subfigure}
    \begin{subfigure}[t]{0.25\textwidth}
        \centering
        \includegraphics[width=\linewidth]{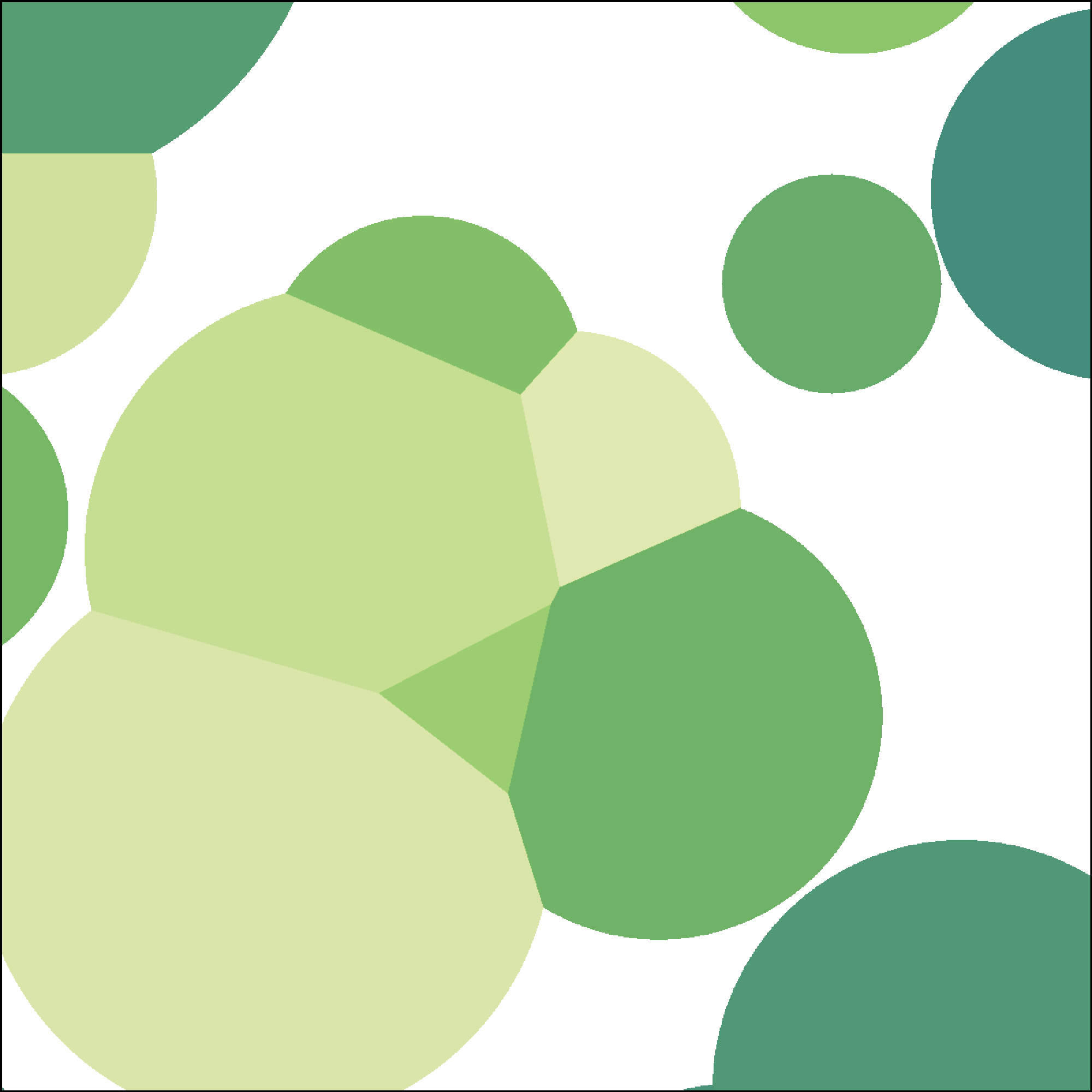}
    \end{subfigure}
    \begin{subfigure}[t]{0.25\textwidth}
        \centering
        \includegraphics[width=\linewidth]{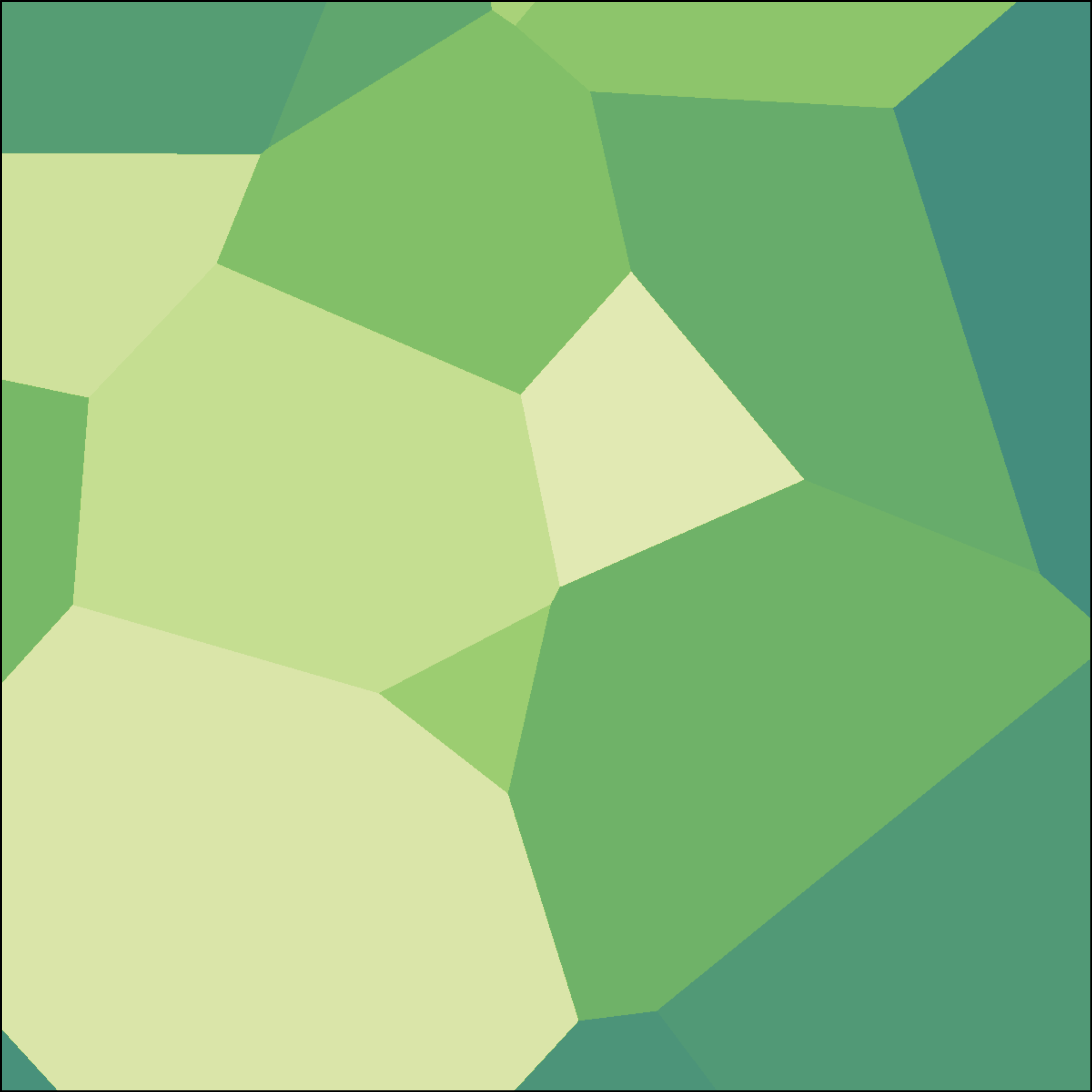}
    \end{subfigure}\hfill}}
    \caption{Visualization of the crystallization process. From left to right, the crystallization process is shown at times $t = 60$, $t=80$, $t=120$ and $t=280$.}
    \label{fig:crystallization}
\end{figure}

Throughout this paper we assume that $\eta$ is a Poisson process on $\RR^d \times (0,\infty)$ with intensity measure $\nu_d \times \mathbb{F}$. Here, $\mathbb{F}$ is a locally finite measure concentrated on $(0,\infty)$. Because the measure $\nu_d \times \mathbb{F}$ has no atoms, $\eta$ is a simple point process. From proposition 3.6. in \cite{Gusakova2024_2} it follows that $L(\eta)$, the Laguerre diagram generated by the Poison process $\eta$, is with probability one a tessellation. We refer to $L(\eta)$ as the Poisson-Laguerre tessellation generated by $\eta$. We do note that in the aforementioned paper it is additionally assumed that $\mathbb{F}$ is absolutely continuous with respect to Lebesgue measure. However, this assumption is not needed for $L(\eta)$ to be a tessellation with probability one, as this is a straightforward modification of the proofs given in \cite{Gusakova2024_2}. 

For $z \geq 0$ we define: $F(z) := \mathbb{F}((0,z])$. Thereby, this monotone function $F$ is the only parameter in this model to be estimated. Note that $F$ is not necessarily bounded, it is bounded if and only if $\mathbb{F}$ is a finite measure. In the introduction of this paper we explained that we are interested in estimators for $F$ which depend on the observed Laguerre cells and the extreme points of $\eta$, which we denote by $\eta^*$:
\[\eta^* := \left\{(x,h) \in \eta: C((x,h),\eta)\neq \emptyset\right\}.\]
To be precise, the estimators we propose for $F$ depend on the points of $\eta^*$ in the observation window $W_n$, as well as the Laguerre cells corresponding to these points of $\eta^*$ in $W_n$. Recall, $(W_n)_{n \geq 1}$ is some fixed convex averaging sequence. The reader may for example keep $W_n = [-n,n]^d$ in mind as an explicit example. Note that the point process $\eta^*$ may be seen as a (dependent) thinning of $\eta$, and is not necessarily a Poisson process. We conclude this section with a simulation example, with the purpose of providing an intuitive understanding of Poisson-Laguerre tessellations. 

\begin{example}\label{example_discrete_laguerre_text}
  In Figure \ref{figure_discrete_laguerre_example} a realization is shown of a planar Poisson-Laguerre tessellation along with its realization of extreme points. Here, we have taken $\mathbb{F}$ to be a discrete probability measure on $\{1,8,10\}$. Specifically, $\mathbb{F}$ is defined as: $\mathbb{F}(\{ 1\}) = 0.01$, $\mathbb{F}(\{8\}) = 0.04$ and $\mathbb{F}(\{10\}) = 0.95$. Hence, $\eta$ may be seen as an independently marked homogeneous Poisson process, with points in $\RR^2$ and marks in $\{1,8,10\}$. The homogeneous Poisson process has intensity 1 and the marks are distributed according to $\mathbb{F}$. Let us briefly discuss the image in Figure \ref{figure_discrete_laguerre_example} in view of the crystallization process interpretation. Given the choice of $\mathbb{F}$, we expect a small number of balls corresponding to points with weight $h=1$, these balls start growing early, and result in large cells. A larger number of points with weight $h=8$ have balls associated with them which start growing later, yielding cells which are a bit smaller. Finally, a very large number of points with weight $h=10$ will generate even smaller cells.   
\end{example}

\section{Inference via a dependent thinning}\label{section_first_estimator}
\subsection{Definition of an estimator}
In this section, we define our first estimator for $F$. This estimator only depends on points $(x,h)$ of $\eta^*$ with $x \in W_n$ and for which $x$ is located in its own Laguerre cell. The estimator is easy to compute, and the techniques used in this section will be important when we define an estimator for $F$ based on all points of $\eta^*$ in $W_n \times (0,\infty)$. Recall from the previous section that $\eta$ is a Poisson process on $\RR^d \times (0,\infty)$, $d \geq 2$, with intensity measure $\nu_d \times \mathbb{F}$. We may also write: $\eta = \{(x_1,h_1), (x_2,h_2),\dots\}$, with $x_i \in \RR^d$, $h_i > 0$. We start as follows, let $y \in \RR^d$, and consider the following thinning of $\eta$:
\begin{equation}
    \eta^y := \left\{(x,h) \in \eta: x+y\in C((x,h),\eta)\right\}. \label{dependent_thinning_def}
\end{equation}
In (\ref{eq_Laguerre_cell_def}) we defined $C((x,h), \eta)$, which denotes the Laguerre cell associated with the weighted point $(x,h) \in \eta$. Evidently, for every $y \in \RR^d$, $\eta^y$ only contains a subset of points of $\eta^*$. Hence, we have: $\eta^y \subset \eta^* \subset \eta$. In particular, for $y=0$ we obtain the set of points of $\eta^*$ which are contained within their own Laguerre cell. In the following lemma we compute the intensity measure of $\eta^y$.

\begin{lemma}\label{dependent_thinning_measure}
    Let $B \in \mathcal{B}(\RR^d)$, $y\in \RR^d$ and $z \geq 0$, the intensity measure $\Lambda^y$ of $\eta^y$ satisfies:
\begin{align*}
    \Lambda^y\left(B \times (0,z] \right) = \nu_d(B)\int_0^z \exp\left(-\kappa_d \int_0^{\Vert y \Vert^2 + h} \left(\Vert y \Vert^2 + h -t \right)^{\frac{d}{2}}\mathrm{d}F(t) \right)\mathrm{d}F(h). 
\end{align*}
\end{lemma}

This intensity measure can be computed via the Mecke equation, which may for example be found in Theorem 4.1 in \cite{Last2018}. The statement is as the follows:

\begin{theorem}[Mecke equation]
    Let $\Lambda$ be a $\sigma$-finite measure on a measurable space $(\mathbb{X}, \mathcal{X})$ and let $\eta$ be a point process on $\mathbb{X}$. Then $\eta$ is a Poisson process with intensity measure $\Lambda$ if and only if:
    \[\EE\left(\sum_{x \in \eta}f\left(x,\eta\right) \right) = \int \EE\left(f\left(x,\eta+\delta_x\right) \right)\Lambda(\mathrm{d}x),\]
    for all non-negative measurable functions $f:\mathbb{X}\times \mathbf{N}(\mathbb{X}) \to [0,\infty]$. 
\end{theorem}

\begin{proof}[Proof of Lemma \ref{dependent_thinning_measure}]
    By definition, the intensity measure of $\eta^y$ is given by:
\begin{align}
    \Lambda^y\left(B \times (0,z] \right) := \EE\left(\eta^y(B \times (0,z]) \right) &= \EE\left(\sum_{(x,h) \in \eta} \mathds{1}_{B}(x)\mathds{1}_{(0,z]}(h) \mathds{1}\left\{x + y \in C((x,h),\eta) \right\} \right). \label{expectation_integrand_mecke}
\end{align}
We rewrite the final indicator function in (\ref{expectation_integrand_mecke}) into a more convenient form. By the definition of a Laguerre cell, we obtain:
\begin{align}
    x + y \in C((x,h), \eta) \iff \Vert y \Vert^2 + h - h' \leq \Vert x + y-x' \Vert^2, \text{ for all } (x',h') \in \eta \iff \eta\left(A_{x,h,y}\right) = 0, \nonumber
\end{align}
where we define the set $A_{x,h,y}$ as:
\[A_{x,h,y} = \left\{(x',h') \in \RR^d \times (0,\infty): \Vert y \Vert^2 + h - h' > \Vert x + y-x' \Vert^2\right\}.\]
Since $\eta$ is a Poisson process, the random variable $\eta\left(A_{x,h,y}\right)$ is Poisson distributed with parameter $\EE(\eta\left(A_{x,h,y}\right))$. As a consequence, the probability that $\eta\left(A_{x,h,y}\right) = 0$ is given by:
\begin{align}
    \PP\left(\eta\left(A_{x,h,y}\right) = 0 \right) &= \exp\left(-\EE\left(\eta\left(A_{x,h,y}\right)\right) \right)\nonumber \\
    &= \exp\left(-\int_{\RR^d}\int_0^{\Vert y \Vert^2 + h} \mathds{1}\left\{\Vert x + y - x'\Vert < \sqrt{\Vert y \Vert^2 + h - t} \right\}\mathrm{d}F(t)\mathrm{d}x' \right)\nonumber\\
    &= \exp\left(-\int_0^{\Vert y \Vert^2 + h}\int_{\RR^d} \mathds{1}\left\{\Vert x'\Vert < \sqrt{\Vert y \Vert^2 + h - t} \right\}\mathrm{d}x'\mathrm{d}F(t) \right) \nonumber\\
    &= \exp\left(-\kappa_d \int_0^{\Vert y \Vert^2 + h} \left(\Vert y \Vert^2 + h -t \right)^{\frac{d}{2}}\mathrm{d}F(t) \right). \label{empty_set_probability}
\end{align}
Note that (\ref{empty_set_probability}) does not depend on $x$. Using (\ref{empty_set_probability}) and the Mecke equation, the expectation in (\ref{expectation_integrand_mecke}) can be computed as follows:
\begin{align}
    \EE&\left(\sum_{(x,h) \in \eta} \mathds{1}_{B}(x)\mathds{1}_{(0,z]}(h) \mathds{1}\{\eta\left(A_{x,h,y}\right) = 0\} \right) = \nonumber \\
    &= \int_0^\infty \int_{\RR^d} \mathds{1}_{B}(x)\mathds{1}_{(0,z]}(h)\PP\left(\eta\left(A_{x,h,y}\right) = 0 \right)\mathrm{d}x\mathrm{d}F(h) \label{equation_mecke_application} \\
    &= \int_0^\infty \int_{\RR^d} \mathds{1}_{B}(x)\mathds{1}_{(0,z]}(h)\exp\left(-\kappa_d \int_0^{\Vert y \Vert^2 + h} \left(\Vert y \Vert^2 + h -t \right)^{\frac{d}{2}}\mathrm{d}F(t) \right)\mathrm{d}x\mathrm{d}F(h) \nonumber \\
    &=  \nu_d(B)\int_0^z \exp\left(-\kappa_d \int_0^{\Vert y \Vert^2 + h} \left(\Vert y \Vert^2 + h -t \right)^{\frac{d}{2}}\mathrm{d}F(t) \right)\mathrm{d}F(h).\nonumber
\end{align}
In (\ref{equation_mecke_application}) we used the fact that $(x,h) \notin A_{x,h,y}$ such that $\eta\left(A_{x,h,y}\right) = (\eta + \delta_{(x,h)})\left(A_{x,h,y}\right)$.
\end{proof}

Recall that $\mathcal{F}_{+}$ denotes the space of all (not necessarily bounded) distribution functions on $(0,\infty)$. Given the statement of Lemma \ref{dependent_thinning_measure} we focus on the case $y=0$ and define for $F \in \mathcal{F}_{+}$ the function $G_F:[0,\infty)\to[0,\infty)$ via:
\begin{equation}
    G_F(z):= \int_0^z \exp\left(-\kappa_d \int_0^{h} \left(h -t \right)^{\frac{d}{2}}\mathrm{d}F(t) \right)\mathrm{d}F(h). \label{G_operator_def}
\end{equation}
For functions $G_F$ with $F \in \mathcal{F}_{+}$ as in (\ref{G_operator_def}) we obtain the following important identifiability result:

\begin{theorem}\label{thm_G_identifiable}
    Let $F_1, F_2 \in \mathcal{F}_{+}$, $R > 0$. If $G_{F_1}(z) = G_{F_2}(z)$ for all $z \in [0,R)$ then $F_1(z) = F_2(z)$ for all $z \in [0,R)$. In particular, if $G_{F_1} = G_{F_2}$ then $F_1 = F_2$.
\end{theorem}

The key ingredient for the proof of this theorem is a variant of the Gr\"onwall inequality. This inequality is in particular known for its applications in integral- and differential equations. We refer to \cite{Pachpatte1998} for more variants of this inequality and their applications.

\begin{theorem}[Theorem 1.3.3. in \cite{Pachpatte1998}]\label{gronwall_thm}
    Suppose $u$, $\alpha$ and $\beta$ are measurable non-negative functions on $[0,\infty)$. Assume that $\alpha$ is non-decreasing. Assume for all $z \geq 0$: $u, \alpha, \beta \in L^1([0,z])$. If for all $z\geq 0$ the following holds:
    \[u(z) \leq \alpha(z) + \beta(z)\int_0^z u(s)\mathrm{d}s.\]
    Then, for all $z\geq 0$:
    \[u(z) \leq \alpha(z)\left(1 + \beta(z)\int_0^z\exp\left(\int_s^z \beta(r)\mathrm{d}r\right)\mathrm{d}s\right).\]
\end{theorem}
Note that if $u,\alpha$ and $\beta$ satisfy the conditions in Theorem \ref{gronwall_thm} and $\beta$ is non-decreasing, then:
\begin{equation}
    u(z) \leq \alpha(z)\left(1+\beta(z)z\exp\left(\beta(z)z\right) \right).\label{gronwall_thm_note}
\end{equation}
We need to point out that in \cite{Pachpatte1998} this theorem also includes the assumption that $u, \alpha$ and $\beta$ are continuous. However, as noted on p. 14 in the same reference, this assumption is not needed. 

\begin{proof}[Proof of Theorem \ref{thm_G_identifiable}]
    Let $z \geq 0$. For $i \in \{1,2\}$ note that the (Lebesgue-Stieltjes) measures associated with $G_{F_i}$ and $F_i$ are mutually absolutely continuous. The corresponding Radon-Nikodym derivative is given by:
    \[\frac{\mathrm{d}G_{F_i}}{\mathrm{d}F_i}(z) = \exp\left(-\kappa_d \int_0^{z} \left(z -t \right)^{\frac{d}{2}}\mathrm{d}F_i(t) \right).\]
    Hence, we may also write:
    \[F_i(z) = \int_0^z \frac{\mathrm{d}F_i}{\mathrm{d}G_{F_i}}(h)\mathrm{d}G_{F_i}(h)  =\int_0^z \exp\left(\kappa_d \int_0^{h} \left(h -t \right)^{\frac{d}{2}}\mathrm{d}F_i(t) \right)\mathrm{d}G_{F_i}(h).\]
    Via integration by parts we may write:
    \[\int_0^z (z-t)^{\frac{d}{2}}\mathrm{d}F_i(t) = 0\cdot F_i(z) - z^{\frac{d}{2}}F_i(0) - \int_0^z F_i(t)\mathrm{d}\left((z-t)^{\frac{d}{2}} \right)(t) = \frac{d}{2}\int_0^z F_i(t)(z-t)^{\frac{d}{2}-1}\mathrm{d}t.\]
    Moreover, 'the expression above' is a non-decreasing function of $z$. We now derive a general upper bound for $|F_1(z)-F_2(z)|$:
    \begin{align}
        |&F_1(z)-F_2(z)| = \nonumber\\ 
        &=\left|\int_0^z \exp\left(\kappa_d \int_0^{h} \left(h -t \right)^{\frac{d}{2}}\mathrm{d}F_1(t) \right)\mathrm{d}G_{F_1}(h) - \int_0^z \exp\left(\kappa_d \int_0^{h} \left(h -t \right)^{\frac{d}{2}}\mathrm{d}F_2(t) \right)\mathrm{d}G_{F_2}(h)\right| \nonumber \\
        \begin{split}\label{two_terms_identifiability_G_proof}
            &\leq \left|\int_0^z \exp\left(\kappa_d \int_0^{h} \left(h -t \right)^{\frac{d}{2}}\mathrm{d}F_1(t) \right)\mathrm{d}G_{F_1}(h) - \int_0^z \exp\left(\kappa_d \int_0^{h} \left(h -t \right)^{\frac{d}{2}}\mathrm{d}F_2(t) \right)\mathrm{d}G_{F_1}(h)\right| + \\
        &\phantom{=} + \left|\int_0^z \exp\left(\kappa_d \int_0^{h} \left(h -t \right)^{\frac{d}{2}}\mathrm{d}F_2(t) \right)\mathrm{d}(G_{F_1}-G_{F_2})(h)\right|.
        \end{split}
    \end{align}
    Let us now consider the first term of (\ref{two_terms_identifiability_G_proof}). For $h \geq 0$ define:
    \[C(h) := \max\left\{\exp\left(\kappa_d \int_0^{h} \left(h -t \right)^{\frac{d}{2}}\mathrm{d}F_1(t)\right),\exp\left(\kappa_d \int_0^{h} \left(h -t \right)^{\frac{d}{2}}\mathrm{d}F_2(t)\right) \right\}.\]
    Note that $C$ is increasing. Since $|e^x - e^y| \leq \max\{e^x,e^y\}|x-y|$ for $x,y\geq 0$ the first term in (\ref{two_terms_identifiability_G_proof}) is bounded by:
    \begin{align*}
        &\text{\quad \ } \int_0^z\left| \exp\left(\kappa_d \int_0^{h} \left(h -t \right)^{\frac{d}{2}}\mathrm{d}F_1(t) \right) - \exp\left(\kappa_d \int_0^{h} \left(h -t \right)^{\frac{d}{2}}\mathrm{d}F_2(t) \right)\right|\mathrm{d}G_{F_1}(h) \\
        &\leq \int_0^z C(h)\kappa_d \left|\int_0^{h} \left(h -t \right)^{\frac{d}{2}}\mathrm{d}F_1(t) - \int_0^{h} \left(h -t \right)^{\frac{d}{2}}\mathrm{d}F_2(t)  \right|\mathrm{d}G_{F_1}(h) \\
        &= \int_0^z C(h)\kappa_d \left|\frac{d}{2}\int_0^h \left(F_1(t) - F_2(t)\right)(h-t)^{\frac{d}{2}-1}\mathrm{d}t \right|\mathrm{d}G_{F_1}(h) \\
        &\leq \frac{d\kappa_d}{2}C(z) \int_0^z \int_0^h \left|F_1(t) - F_2(t)\right|(h-t)^{\frac{d}{2}-1}\mathrm{d}t \mathrm{d}G_{F_1}(h) \\
        &\leq  \frac{d\kappa_d}{2}C(z)z^{\frac{d}{2}-1} \int_0^z \int_0^z \left|F_1(t) - F_2(t)\right|\mathrm{d}t \mathrm{d}G_{F_1}(h)\\
        &= \frac{d\kappa_d}{2}C(z)z^{\frac{d}{2}-1}G_{F_1}(z)\int_0^z \left|F_1(t) - F_2(t)\right|\mathrm{d}t.
    \end{align*}
Via integration by parts, the second term of (\ref{two_terms_identifiability_G_proof}) is bounded by:
    \begin{align*}
         \begin{split}
             &\phantom{\leq} \left|\exp\left(\kappa_d \int_0^{z} \left(z -t \right)^{\frac{d}{2}}\mathrm{d}F_2(t) \right)\left(G_{F_1}(z) - G_{F_2}(z)\right)\right| + \\ 
             &\phantom{=} + \left|\int_0^z \left(G_{F_1}(h) - G_{F_2}(h)\right)\mathrm{d}\left(\exp\left(\kappa_d \int_0^{h} \left(h -t \right)^{\frac{d}{2}}\mathrm{d}F_2(t) \right) \right)(h) \right|
         \end{split}\\
         \begin{split}
             &\leq \left|G_{F_1}(z) - G_{F_2}(z)\right|\exp\left(\kappa_d \int_0^{z} \left(z -t \right)^{\frac{d}{2}}\mathrm{d}F_2(t) \right) + \\ 
             &\phantom{=} + \sup_{h \in [0,z]}\left|G_{F_1}(h) - G_{F_2}(h)\right| \int_0^z \mathrm{d}\left(\exp\left(\kappa_d \int_0^{h} \left(h -t \right)^{\frac{d}{2}}\mathrm{d}F_2(t) \right) \right)(h)
         \end{split}\\
         &\leq \sup_{h \in [0,z]}\left|G_{F_1}(h) - G_{F_2}(h)\right| 2 \exp\left(\kappa_d \int_0^{z} \left(z -t \right)^{\frac{d}{2}}\mathrm{d}F_2(t) \right).
    \end{align*}
    Combining all results, we obtain:
    \begin{align*}
        \begin{split}
            |F_1(z) - F_2(z)| &\leq \frac{d\kappa_d}{2}C(z)z^{\frac{d}{2}-1}G_{F_1}(z)\int_0^z \left|F_1(t) - F_2(t)\right|\mathrm{d}t + \\
            &\phantom{=} + \sup_{h \in [0,z]}\left|G_{F_1}(h) - G_{F_2}(h)\right|2 \exp\left(\kappa_d \int_0^{z} \left(z -t \right)^{\frac{d}{2}}\mathrm{d}F_2(t) \right).
        \end{split} 
    \end{align*}
    Applying Theorem \ref{gronwall_thm} and (\ref{gronwall_thm_note}) with $u(z) = |F_1(z) - F_2(z)|$ yields:
    \begin{equation}
        |F_1(z) - F_2(z)| \leq 
        K(z)\sup_{h \in [0,z]}\left|G_{F_1}(h) - G_{F_2}(h)\right| . \label{G_identfiable_proof_final_eq}
    \end{equation}
    Here, $K(z)$ is given by:
    \[K(z) := \left(1+ \frac{d\kappa_d}{2}C(z)z^{\frac{d}{2}}G_{F_1}(z) \exp \left(\frac{d\kappa_d}{2}C(z)z^{\frac{d}{2}}G_{F_1}(z)\right)\right)2 \exp\left(\kappa_d \int_0^{z} \left(z -t \right)^{\frac{d}{2}}\mathrm{d}F_2(t) \right).\]
    The statement of the theorem immediately follows from (\ref{G_identfiable_proof_final_eq}).
\end{proof}

Suppose we wish to estimate $G_F$, and we observe the extreme points of $\eta$ within the bounded observation window $W_n$, as well as their Laguerre cells. We define the following unbiased estimator for $G_F$:
\begin{align}
    \hat{G}_n(z) :&= \frac{1}{\nu_d(W_n)}\sum_{(x,h) \in \eta} \mathds{1}_{W_n}(x)\mathds{1}_{(0,z]}(h) \mathds{1}\{x \in C((x,h),\eta) \} \nonumber\\
    &= \frac{1}{\nu_d(W_n)}\sum_{(x,h) \in \eta^0} \mathds{1}_{W_n}(x)\mathds{1}_{(0,z]}(h).\label{G_estimator_def}
\end{align}
Hence, $G_F $ is a function which we can estimate and which uniquely determines $F$, this motivates the following definition:

\begin{definition}[First inverse estimator of $F$]\label{inverse_estimator_F0_def}
    Define $\hat{F}_n^0$ to be the unique function $\hat{F}_n^0 \in \mathcal{F}_{+}$ which satisfies: $G_{\smash{\hat{F}_n^0}}(z) = \hat{G}_n(z)$ for all $z \geq 0$, with $\hat{G}_n$ as in (\ref{G_estimator_def}).
\end{definition}

Let us now discuss why $\hat{F}_n^0$ is well-defined. Clearly, if there exists a function $\hat{F}_n^0 \in \mathcal{F}_{+}$ which satisfies $G_{\smash{\hat{F}_n^0}}(z) = \hat{G}_n(z)$ for all $z \geq 0$ then it is unique by Theorem \ref{thm_G_identifiable}. Suppose $(x_1,h_1),(x_2,h_2),\dots,\allowbreak(x_k,h_k)$ is the sorted realization of the points of $\eta^0$ with $x_1,\dots,x_k \in W_n$ and $h_1 \leq h_2 \leq \dots \leq h_k$. We may write:
\[\hat{G}_n(z) = \frac{1}{\nu_d(W_n)}\sum_{i=1}^k \mathds{1}\{h_i \leq z\}.\]
Set $h_0 = 0$ such that $\hat{F}_n^0(h_0)= 0$. Clearly, $\hat{G}_n$ is piecewise constant, with jump locations at $h_1,\dots,h_k$. Recall from the proof of Theorem \ref{thm_G_identifiable} that the Lebesgue-Stieltjes measures associated with $\hat{F}_n^0$ and $G_{\smash{\hat{F}_n^0}}$ are mutually absolutely continuous. As a consequence, if $\hat{F}_n^0$ exists, it is necessarily also piecewise constant with the same jump locations as $G_{\smash{\hat{F}_n^0}} = \hat{G}_n$. Therefore, if we can uniquely specify the value of $\hat{F}_n^0$ at $h_1,\dots,h_k$, existence and uniqueness of $\hat{F}_n^0$ is established. Let $i \in \{1,\dots,k\}$ then, for the $\hat{F}_n^0$ we are looking for:
\begin{align*}
    \hat{G}_n(h_i) &= \hat{G}_n(h_{i-1}) + \int_{h_{i-1}}^{h_i} \exp\left(-\kappa_d \int_0^{h} \left(h -t \right)^{\frac{d}{2}}\mathrm{d}\hat{F}_n^0(t) \right)\mathrm{d}\hat{F}_n^0(h) \\
    &= \hat{G}_n(h_{i-1}) + \exp\left(-\kappa_d \sum_{j=1}^i \left(h_i -h_j \right)^{\frac{d}{2}}\left(\hat{F}_n^0(h_j) - \hat{F}_n^0(h_{j-1}) \right) \right)\left(\hat{F}_n^0(h_i) - \hat{F}_n^0(h_{i-1}) \right)\\
    &= \hat{G}_n(h_{i-1}) + \exp\left(-\kappa_d \sum_{j=1}^{i-1} \left(h_i -h_j \right)^{\frac{d}{2}}\left(\hat{F}_n^0(h_j) - \hat{F}_n^0(h_{j-1}) \right) \right)\left(\hat{F}_n^0(h_i) - \hat{F}_n^0(h_{i-1}) \right).
\end{align*}
Since $\hat{F}_n^0(h_0)= 0$, $\hat{F}_n^0$ is recursively defined via:
\begin{align}
    \hat{F}_n^0(h_i) &= \hat{F}_n^0(h_{i-1}) + \left(\hat{G}_n(h_i) - \hat{G}_n(h_{i-1}) \right)\exp\left(\kappa_d \sum_{j=1}^{i-1} \left(h_i -h_j \right)^{\frac{d}{2}}\left(\hat{F}_n^0(h_j) - \hat{F}_n^0(h_{j-1}) \right) \right). \label{computational_formula_F_estimator}
\end{align}
Note that the RHS of (\ref{computational_formula_F_estimator}) only depends on the values $\hat{F}_n^0(h_j)$ with $j < i$. So indeed, (\ref{computational_formula_F_estimator}) completely defines $\hat{F}_n^0$. Moreover, this expression is also a convenient formula for computing $\hat{F}_n^0$ in practice.

\subsection{Consistency}

In this section we show that $\hat{F}_n^0$, as in Definition \ref{inverse_estimator_F0_def}, is a strongly consistent estimator for $F$. The first step is to show that the estimator $\hat{G}_n$ as in (\ref{G_estimator_def}) for $G_F$ is strongly consistent. For empirical estimators such as $\hat{G}_n$, their consistency follows from a spatial ergodic theorem. From Proposition 13.4.I. in \cite{Daley2008}, and the ergodicity of the Poisson process under consideration, we obtain:

\begin{theorem}[Spatial ergodic theorem]
    Let $\eta$ be a Poisson process on $\mathbb{X} = \RR^d \times (0,\infty)$ with intensity measure $\nu_d \times \mathbb{F}$. Here, $\mathbb{F}$ is a locally finite measure concentrated on $(0, \infty)$. Let $g(\psi,h)$ be a measurable non-negative function on $\mathbf{N}(\mathbb{X}) \times (0, \infty)$. Then, for any convex averaging sequence $(W_n)_{n \geq 1}$:
    \[\lim_{n \to \infty} \frac{1}{\nu_d(W_n)}\sum_{(x,h) \in \eta}\mathds{1}_{W_n}(x)g(S_x \eta, h) \overset{\text{a.s.}}{=} \int_0^\infty \EE\left(g\left(\eta + \delta_{(0,h)},h \right) \right)\mathbb{F}(\mathrm{d}h).\]
\end{theorem}

We do note that Proposition 13.4.I in \cite{Daley2008} is phrased in the context that $\mathbb{F}$ is a finite measure. However, like Theorem 12.2.IV in the same reference (another spatial ergodic theorem), which is stated under the assumption that $\mathbb{F}$ is locally finite, the result remains valid if $\mathbb{F}$ is locally finite. Besides the spatial ergodic theorem we also need the following useful lemma for estimators of monotone functions:
\begin{lemma}\label{monotone_estimator_lemma}
    Let $(F_n)_{n \geq 1}$ be a random sequence of monotone functions on $\RR$, and let $F$ be a deterministic monotone function on $\RR$. If for all $z \in \RR$: $\PP(\lim_{n \to \infty} F_n(z) = F(z)) = 1$, then: $\PP(\lim_{n \to \infty} F_n(z) = F(z), \ \forall z \in \RR) = 1$.
\end{lemma}

The proof of Lemma \ref{monotone_estimator_lemma} is given in Appendix \ref{appendix_proofs}. We obtain the following result:

\begin{corollary}\label{corr_consistency_Gn}
    With probability one: $\lim_{n \to \infty} \hat{G}_n(z) = G_F(z)$ for all $z \geq 0$.
\end{corollary}

\begin{proof}
    Let $z \geq 0$, by Lemma \ref{monotone_estimator_lemma} it is sufficient to show that $\lim_{n \to \infty} \hat{G}_n(z) = G_F(z)$ almost surely. Using the same notation as in the proof of Lemma \ref{dependent_thinning_measure}, note that $\eta(A_{x,h,0})=S_x \eta(A_{0,h,0})$ for all $(x,h) \in \eta$ almost surely. As a consequence, $\hat{G}_n(z)$ may be written as follows:
    \begin{align*}
        \hat{G}_n(z) = \frac{1}{\nu_d(W_n)}\sum_{(x,h) \in \eta} \mathds{1}_{W_n}(x)\mathds{1}_{(0,z]}(h) \mathds{1}\{S_x\eta(A_{0,h,0}) = 0\}.
    \end{align*}
    Following the computation in the proof of Lemma \ref{dependent_thinning_measure}, it is readily verfied that applying the spatial ergodic theorem with $g(\psi,h) = \mathds{1}_{(0,z]}(h)\mathds{1}\{\psi(A_{0,h,0}) = 0\}$ yields the result with the desired limit.
\end{proof}

Finally, we need the following continuity result:

\begin{lemma}\label{continuity_F_in_G}
    Let $(F_n)_{n \geq 1}$ be a sequence of functions in $\mathcal{F}_{+}$ and let $F \in \mathcal{F}_{+}$. Let $R > 0$. If\\ $\lim_{n \to \infty} F_n(z) = F(z)$ for all $z \in [0,R)$, then $\lim_{n \to \infty} G_{F_n}(z) = G_F(z)$ for all $z \in [0,R)$. In particular, if $\lim_{n \to \infty} F_n(z) = F(z)$ for all $z \geq 0$, then $\lim_{n \to \infty} G_{F_n}(z) = G_F(z)$ for all $z \geq 0$. 
\end{lemma}

The proof of Lemma \ref{continuity_F_in_G} is given in Appendix \ref{appendix_proofs}. Combining the previous results with Theorem \ref{thm_G_identifiable} we prove the following consistency result.

\begin{figure}[b!]
    \centering
    \makebox[\textwidth]{\makebox[\textwidth]{
    \begin{subfigure}[t]{0.5\textwidth}
        \centering
        \includegraphics{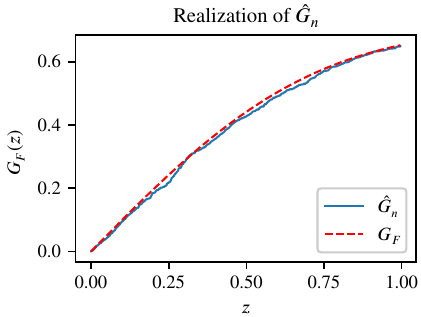}
    \end{subfigure}
    \begin{subfigure}[t]{0.5\textwidth}
        \centering
        \includegraphics{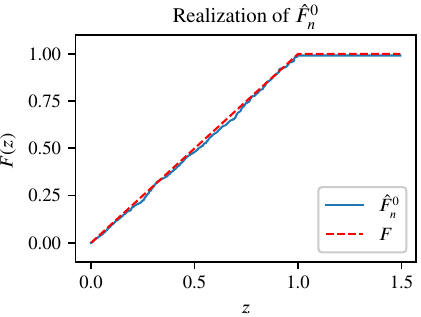}
    \end{subfigure}\hfill}}
    \caption{Left: A realization of $\hat{G}_n$. Right: The corresponding realization of $\hat{F}_n^0$. The actual underlying $F$ is equal to the CDF of a uniform distribution on $(0, 1)$.}
    \label{figure_example_simulation_Fn0}
\end{figure}

\begin{theorem}[Consistency of $\hat{F}_n^0$]
    With probability one, $\lim_{n \to \infty}\hat{F}_n^0(z) = F(z)$ for all $z \geq 0$.
\end{theorem}

\begin{proof}
Let $(\Omega, \mathcal{A},\PP)$ be a probability space supporting a Poisson process $\eta$, with intensity measure $\nu_d \times \mathbb{F}$. By Corollary \ref{corr_consistency_Gn} there exists a set $\Omega_0 \in \mathcal{A}$ with $\PP(\Omega_0)=1$ such that for all $\omega \in \Omega_0$ and $z \geq 0$ we have $\lim_{n \to \infty}\hat{G}_n(z;\omega) = G_F(z)$. Let $z \geq 0$, we show that $\lim_{n \to \infty}\hat{F}_n^0(z;\omega) = F(z)$.

Pick $M > 0$ such that $F(z) < M$. For $n \in \NN$ and $h \geq 0$, define: $\bar{F}_n(h) = \min\{\hat{F}_n^0(h;\omega), M\}$. Then, $(\bar{F}_n)_{n \geq 1}$ is a uniformly bounded sequence of monotone functions. Let $(n_l)_{l \geq 1} \subset (n)_{n\geq 1}$ be an arbitrary subsequence. By Helly's selection principle there exists a further subsequence $(n_k)_{k\geq 1} \subset (n_l)_{l \geq 1}$ such that $\bar{F}_{n_k}$ converges pointwise to some monotone function $\bar{F}$ as $k \to \infty$. This implies that $\lim_{k \to \infty} \hat{F}_{n_k}^0(h;\omega) = \lim_{k \to \infty}\bar{F}_{n_k}(h) = \bar{F}(h)$ for all $h \in [0,R)$ with $R:= \sup\{h\geq 0: \bar{F}(h) < M\}$. By Lemma \ref{continuity_F_in_G} we obtain: 
\[\lim_{k \to \infty}\hat{G}_{n_k}(h;\omega) := \lim_{k \to \infty}G_{\hat{F}_{n_k}^0(\blank;\omega)}(h) = G_{\bar{F}}(h) \text{ \ for all } h \in [0,R).\]
Because the whole sequence $\hat{G}_n(h;\omega)$ converges to $G_F(h)$ as $n \to \infty$, for $h \geq 0$, we obtain $G_F(h) = G_{\bar{F}}(h)$ for all $h \in [0,R)$. Theorem \ref{thm_G_identifiable} now yields $F(h) = \bar{F}(h)$ for all $h \in [0,R)$, and since $z \in [0,R)$ we have in particular $F(z) = \bar{F}(z)$. As a consequence: $\lim_{k \to \infty}\smash{\hat{F}_{n_k}^0}(z;\omega) = F(z)$. Because the initial subsequence was chosen arbitrarily, the whole sequence converges: $\lim_{n \to \infty}\smash{\hat{F}_n^0}(z;\omega) = F(z)$.
\end{proof}

In Figure \ref{figure_example_simulation_Fn0} a single realization of $\hat{G}_n$ and $\hat{F}_n^0$ are shown. We present additional simulation results in section \ref{section_simulations}.

\section{Inference via the volume-biased weight distribution}\label{section_second_estimator}
\subsection{Definition of an estimator}
In this section we define a second estimator for $F$, which depends on all points of $\eta^*$ in $W_n \times (0,\infty)$ as well as the volumes of the Laguerre cells corresponding to these points. As such, this estimator depends on more data compared to the estimator in the previous section. First, we present a result for Poisson-Laguerre tessellations in $\RR^d$, the estimator itself is defined specifically for the planar case ($d=2$). Suppose for now that $\mathbb{F}$ is a finite measure, such that $\eta$ may be interpreted as an independently marked Poisson process. Because $\mathbb{F}$ then determines the distribution of the weights ($h$-coordinates) of the points of $\eta$, a natural question is to ask how the distribution of the weights of the points of $\eta^*$ is related to $\mathbb{F}$. As it turns out, it is more tractable to study a biased or weighted version of this distribution. We introduce the so-called volume-biased weight distribution in the following definition, which is also well-defined if $\mathbb{F}$ is not a finite measure:

\begin{definition}[volume-biased weight distribution]\label{def_biased_weight_dist}
    Let $\eta$ be a Poisson process on $\RR^d \times (0,\infty)$, $d \geq 2$, with intensity measure $\nu_d \times \mathbb{F}$. Here, $\mathbb{F}$ is a locally finite measure concentrated on $(0,\infty)$. Let $A \in \mathcal{B}(\RR)$, define the following probability measure:
    \begin{equation}
        \mathbb{F}^V(A) := \EE\left(\sum_{(x,h) \in \eta} \mathds{1}_{[0,1]^d}(x)\mathds{1}_A(h)\nu_d\left(C((x,h),\eta)\right) \right). \label{volume_biased_weight_dist_def_eq}
    \end{equation}
\end{definition}

Consider the Poisson-Laguerre tessellation generated by $\eta$, then the interpretation of $\mathbb{F}^V$ is as follows. $\mathbb{F}^V$ describes the distribution of the random weight associated with a randomly chosen Laguerre cell, the probability of picking any given cell being proportional to its volume. Because there is an infinite number of Laguerre cells in the tessellation, care needs to be taken in making this statement precise. This can be done via Palm calculus for marked point processes, see for instance chapter 3 in \cite{Schneider2008}. Note that the sum in (\ref{volume_biased_weight_dist_def_eq}) effectively only sums over points $(x,h) \in \eta$ with a Laguerre cell $C((x,h),\eta)$ of positive volume. Hence, it can also be seen as a sum over points of $\eta^*$. From its definition it is not immediately obvious that $\mathbb{F}^V$ is a well-defined probability measure. Specifically, it is not immediately evident that $\mathbb{F}^V(\RR) = 1$. We address this in the proof of the following theorem, where we derive the CDF (Cumulative Distribution Function) associated with $\mathbb{F}^V$.

\begin{theorem}\label{weighted_height_dist_thm}
    Let $\eta$ be a Poisson process as in Definition \ref{def_biased_weight_dist}, and let $z \geq 0$. Define $F(z) := \mathbb{F}((0,z])$ and $F^V(z) := \mathbb{F}^V((0,z])$, the distribution functions corresponding to $\mathbb{F}$ and $\mathbb{F}^V$ respectively. The measure $\mathbb{F}^V$ is a probability measure and $F^V$ is given by:
    \begin{align*}
        F^V(z) &= 1 - \exp\left(-\kappa_d \int_0^z (z-t)^{\frac{d}{2}}\mathrm{d}F(t)\right) + \\
         &\phantom{=} \text{ \ } + \frac{\omega_d}{2} \int_z^\infty  \exp\left(-\kappa_d \int_0^{u} \left(u -t \right)^{\frac{d}{2}}\mathrm{d}F(t) \right)\int_0^z (u-h)^{\frac{d}{2}-1} \mathrm{d}F(h)\mathrm{d}u.
    \end{align*}
\end{theorem}

\begin{proof}
    By the translation invariance of Lebesgue measure and Fubini's theorem, we may write:
    \begin{align}
        F^V(z) &= \EE\left(\sum_{(x,h) \in \eta} \mathds{1}_{[0,1]^d}(x)\mathds{1}_{(0,z]}(h)\nu_d\left(C((x,h),\eta) - x \right) \right) \nonumber \\
        &= \EE\left(\sum_{(x,h) \in \eta} \mathds{1}_{[0,1]^d}(x)\mathds{1}_{(0,z]}(h) \int_{\RR^d}\mathds{1}\{y \in C((x,h),\eta) - x \}\mathrm{d}y \right)\nonumber \\
        &= \int_{\RR^d}\EE\left(\sum_{(x,h) \in \eta} \mathds{1}_{[0,1]^d}(x)\mathds{1}_{(0,z]}(h) \mathds{1}\{x + y \in C((x,h),\eta) \} \right)\mathrm{d}y \nonumber \\
        &= \int_{\RR^d}\EE\left(\eta^y([0,1]^d \times (0,z]) \right)\mathrm{d}y. \label{volume_weighted_proof_measure_eq}
    \end{align}
With $\eta^y$ as in (\ref{dependent_thinning_def}). In Lemma \ref{dependent_thinning_measure} we computed the expectation in (\ref{volume_weighted_proof_measure_eq}). Plugging in this expression, and passing to polar coordinates by substituting $y = r\theta$, with $r \geq 0$ and $\theta \in \Sp^{d-1}$, we obtain:
\begin{align}
    F^V(z) &= \int_{\RR^d} \int_0^z \exp\left(-\kappa_d \int_0^{\Vert y \Vert^2 + h} \left(\Vert y \Vert^2 + h -t \right)^{\frac{d}{2}}\mathrm{d}F(t) \right)\mathrm{d}F(h) \mathrm{d}y  \nonumber\\
    %&= \omega_d \int_0^\infty \int_0^z \exp\left(-\kappa_d \int_0^{r^2 + h} \left(r^2 + h -t \right)^{\frac{d}{2}}\mathrm{d}F(t) \right)\mathrm{d}F(h) r^{d-1}\mathrm{d}r   \nonumber \\
    &= \omega_d \int_0^z \int_0^\infty \exp\left(-\kappa_d \int_0^{r^2 + h} \left(r^2 + h -t \right)^{\frac{d}{2}}\mathrm{d}F(t) \right)r^{d-1}\mathrm{d}r \mathrm{d}F(h) \label{FV_proof_eq1} \\
    &= \frac{\omega_d}{2} \int_0^z \int_h^\infty \exp\left(-\kappa_d \int_0^{u} \left(u -t \right)^{\frac{d}{2}}\mathrm{d}F(t) \right)(u-h)^{\frac{d}{2}-1}\mathrm{d}u \mathrm{d}F(h) \label{FV_proof_eq2} \\
    &= \frac{\omega_d}{2} \int_0^\infty  \exp\left(-\kappa_d \int_0^{u} \left(u -t \right)^{\frac{d}{2}}\mathrm{d}F(t) \right)\int_0^{\min\{u,z\}}(u-h)^{\frac{d}{2}-1} \mathrm{d}F(h)\mathrm{d}u.  \label{cdf_convenient_form}
\end{align}
In (\ref{FV_proof_eq1}) and (\ref{cdf_convenient_form}) we apply Fubini's theorem, and in (\ref{FV_proof_eq2}) we substitute $u = r^2+h$. We can now write $F^V(z)$ as a sum of two integrals:
\begin{align}
    \begin{split}\label{FV_proof_integral_split}
    F^V(z) &= \frac{\omega_d}{2} \int_0^z  \exp\left(-\kappa_d \int_0^{u} \left(u -t \right)^{\frac{d}{2}}\mathrm{d}F(t) \right)\int_0^u (u-h)^{\frac{d}{2}-1} \mathrm{d}F(h)\mathrm{d}u +  \\
    &\phantom{=} \text{ \ } + \frac{\omega_d}{2} \int_z^\infty  \exp\left(-\kappa_d \int_0^{u} \left(u -t \right)^{\frac{d}{2}}\mathrm{d}F(t) \right)\int_0^z (u-h)^{\frac{d}{2}-1} \mathrm{d}F(h)\mathrm{d}u.   
    \end{split}
\end{align}
The first integral of (\ref{FV_proof_integral_split}) can be calculated explicitly since the integrand has an explicit primitive. Using the fact $\omega_d = d \kappa_d$, the first term is given by:
\[\left[ -\exp\left(-\kappa_d \int_0^{u} \left(u -t \right)^{\frac{d}{2}}\mathrm{d}F(t) \right)\right]_0^z = 1 - \exp\left(-\kappa_d \int_0^z (z-t)^{\frac{d}{2}}\mathrm{d}F(t)\right).\]
Plugging this back into (\ref{FV_proof_integral_split}) yields the expression for $F^V$ as stated in the theorem. Finally, via (\ref{cdf_convenient_form}) we can show that $\lim_{z \to \infty} F^V(z) = 1$. After all, the integrand in (\ref{cdf_convenient_form}) (considering the integral w.r.t. $u$) can be bounded from above using the inequality $\min\{u,z\} \leq u$. Via the dominated convergence theorem it follows that:
\begin{align*}
    \lim_{z \to \infty} F^V(z) &= \frac{\omega_d}{2} \int_0^\infty   \exp\left(-\kappa_d \int_0^{u} \left(u -t \right)^{\frac{d}{2}}\mathrm{d}F(t) \right)\lim_{z \to \infty}\int_0^{\min\{u,z\}}(u-h)^{\frac{d}{2}-1} \mathrm{d}F(h)\mathrm{d}u\\
    &= \int_0^\infty  \exp\left(-\kappa_d \int_0^{u} \left(u -t \right)^{\frac{d}{2}}\mathrm{d}F(t) \right)\frac{d \kappa_d}{2}\int_0^u (u-h)^{\frac{d}{2}-1} \mathrm{d}F(h)\mathrm{d}u \\
    &= \left[ -\exp\left(-\kappa_d \int_0^{u} \left(u -t \right)^{\frac{d}{2}}\mathrm{d}F(t) \right)\right]_0^\infty = 1.
\end{align*}
\end{proof}

The Stieltjes integrals in the expression for $F^V$ may be written as Lebesgue integrals using integration by parts. For instance:
\begin{equation}
    \int_0^z (z-t)^{\frac{d}{2}}\mathrm{d}F(t) = 0\cdot F(z) - z^{\frac{d}{2}}F(0) - \int_0^z F(t)\mathrm{d}\left((z-t)^{\frac{d}{2}} \right)(t) = \frac{d}{2}\int_0^z F(t)(z-t)^{\frac{d}{2}-1}\mathrm{d}t.\label{integration_by_parts}
\end{equation}
As announced in the beginning of this section, we will now focus on the case $d=2$, which is important for practical applications. In that case, Theorem \ref{weighted_height_dist_thm} and (\ref{integration_by_parts}) yield the following expression for $F^V$.

\begin{corollary}
    Let $z \geq 0$, if $d=2$ the CDF $F^V$ is given by:
    \begin{equation}
        F^V(z) = 1 - \exp\left(-\pi \int_0^z F(t)\mathrm{d}t\right) + \pi F(z)\int_z^\infty  \exp\left(-\pi \int_0^{u} F(t)\mathrm{d}t \right)\mathrm{d}u.\label{FV_planar_expression}
    \end{equation}
\end{corollary}

Let us now introduce some convenient notation which will be used throughout this section. For $z \geq 0$, $F \in \mathcal{F}_{+}$ and $m \geq 0$ we define:
\begin{align*}
    V(z;F,m) &:= 1 - \exp\left(-\pi \int_0^z F(t)\mathrm{d}t\right) + \pi F(z)\left(m - \int_0^z  \exp\left(-\pi \int_0^{u} F(t)\mathrm{d}t \right)\mathrm{d}u\right).\\
    m_F &:= \int_0^\infty  \exp\left(-\pi \int_0^{u} F(t)\mathrm{d}t \right)\mathrm{d}u.
\end{align*}
Note that if $m = m_F$, then $V(\blank;F,m) = F^V$, with $F^V$ as in (\ref{FV_planar_expression}). In other words, $V(\blank;F,m)$ is then the volume-biased weight distribution induced by $F$. We obtain the following identifiability result:

\begin{theorem}\label{FV_identifiability_theorem}
    Let $F_1, F_2 \in \mathcal{F}_{+}$, let $R > 0$. If $m_{F_1} = m_{F_2}$ and $V(z;F_1,m_{F_1}) = V(z;F_2,m_{F_2})$ for all $z \in [0,R)$, then $F_1(z) = F_2(z)$ for all $z \in [0,R)$. Consequently, if $m_{F_1} = m_{F_2}$ and $V(\blank;F_1,m_{F_1}) = V(\blank;F_2,m_{F_2})$, then $F_1 = F_2$.
\end{theorem}

The proof of Theorem \ref{FV_identifiability_theorem} as well as the proofs of most of the remaining lemma's in this section are postponed to Appendix \ref{appendix_proofs}. The techniques used for proving these results are similar to the techniques used in section \ref{section_first_estimator}. We now define the following natural estimator for the distribution function $F^V$:
\begin{equation}
    \tilde{F}_n^V(z) := \frac{1}{\nu_d(W_n)}\sum_{(x,h) \in \eta} \mathds{1}_{W_n}(x)\mathds{1}_{(0,z]}(h)\nu_d\left(C((x,h),\eta)\right).\label{Fv_estimator_def1}
\end{equation}
Alternatively, the following estimator for $F^V$ may be defined:
\begin{equation}
    \hat{F}_n^V(z) := \frac{\sum_{(x,h) \in \eta} \mathds{1}_{W_n}(x)\mathds{1}_{(0,z]}(h)\nu_d\left(C((x,h),\eta)\right)}{\sum_{(x,h) \in \eta} \mathds{1}_{W_n}(x)\nu_d\left(C((x,h),\eta)\right)}.\label{Fv_estimator_def2}
\end{equation}
\begin{remark}
    Note that the estimators $\hat{F}_n^V$ and $\tilde{F}_n^V$ for $F^V$ do not incorporate edge effects. For instance, a Laguerre cell may be partially observed through the observation window $W_n$, such that computation of the estimators requires information outside of the window. In practice one could artificially shrink the observation window such that the estimators can be computed based on this smaller window.
\end{remark}

Similarly to $\hat{F}_n^0$, we can define an inverse estimator for $F$ using an estimator for $F^V$. We choose to use $\hat{F}_n^V$ for this purpose, since it satisfies $\lim_{z \to \infty} \hat{F}_n^V(z) = 1$, in general this is not the case for $\tilde{F}_n^V$.  In view of Theorem \ref{FV_identifiability_theorem} we need to keep in mind that the constant $m_F$ is unknown. We can resolve this by first using $\hat{F}_n^0$ to estimate $m_F$. That is, we define:
\begin{equation}
    \hat{m}_n := m_{\hat{F}_n^0} = \int_0^\infty  \exp\left(-\pi \int_0^{u} \hat{F}_n^0(t)\mathrm{d}t \right)\mathrm{d}u. \label{mf_estimator}
\end{equation}
Finally, we define our second estimator for $F$ as follows:

\begin{definition}[Second inverse estimator of F]\label{inverse_estimator_F_def}
    Define $\hat{F}_n$ to be the unique function $\hat{F}_n \in \mathcal{F}_{+}$ which satisfies for all $z \geq 0$: 
    \begin{equation}
        \hat{F}_n^V(z) = 1 - \exp\left(-\pi \int_0^z \hat{F}_n(t)\mathrm{d}t\right) + \pi \hat{F}_n(z)\left(\hat{m}_n - \int_0^z  \exp\left(-\pi \int_0^{u} \hat{F}_n(t)\mathrm{d}t \right)\mathrm{d}u\right), \label{inverse_estimator_F_eq}
    \end{equation}
    with $\hat{F}_n^V$ as in (\ref{Fv_estimator_def2}) and $\hat{m}_n$ as in (\ref{mf_estimator}). That is, $\hat{F}_n$ is the unique function $\hat{F}_n \in \mathcal{F}_{+}$ which satisfies for all $z \geq 0$: $V(z;\hat{F}_n,\hat{m}_n) = \hat{F}_n^V(z)$.
\end{definition}

We again discuss why $\hat{F}_n$ is well-defined. If there exists a function $\hat{F}_n \in \mathcal{F}_{+}$ which satisfies $V(z;\hat{F}_n,\hat{m}_n) = \smash{\hat{F}_n^V(z)}$ for all $z \geq 0$ then it is unique by Theorem \ref{FV_identifiability_theorem}. From (\ref{inverse_estimator_F_eq}) we see that $\hat{F}_n$ cannot be the zero function. Moreover, we see that $\hat{F}_n$ should satisfy the following:
\[m_{\hat{F}_n} = \lim_{z \to \infty}\int_0^z  \exp\left(-\pi \int_0^{u} \hat{F}_n(t)\mathrm{d}t \right)\mathrm{d}u = \hat{m}_n - \lim_{z \to \infty} \frac{\hat{F}_n^V(z) - 1 + \exp\left(-\pi \int_0^z \hat{F}_n(t)\mathrm{d}t\right)}{\pi \hat{F}_n(z)} = \hat{m}_n.\]
The final equality follows from the fact that $\lim_{z \to \infty}\hat{F}_n^V(z) = 1$ and $\lim_{z \to \infty}\hat{F}_n(z) > 0$, since $\hat{F}_n$ is non-zero. Therefore, $\hat{F}_n$ necessarily satisfies:
\begin{equation}
    \hat{F}_n^V(z) = 1 - \exp\left(-\pi \int_0^z \hat{F}_n(t)\mathrm{d}t\right) + \pi \hat{F}_n(z)\int_z^\infty  \exp\left(-\pi \int_0^{u} \hat{F}_n(t)\mathrm{d}t \right)\mathrm{d}u. \label{inverse_estimator_induced_FV}
\end{equation}
Recall from (\ref{FV_planar_expression}) that this means that $\hat{F}_n^V$ is the volume-biased weight distribution induced by $\hat{F}_n$. Suppose $(x_1,h_1),(x_2,h_2),\dots,\allowbreak(x_m,h_k)$ is the sorted realization of the points of $\eta^*$ with $x_1,\dots,x_k \in W_n$ and $h_1 \leq h_2 \leq \dots \leq h_k$. 
Clearly, $\smash{\hat{F}_n^V(z)}$ is piecewise constant, with jump locations at $h_1,\dots,h_k$. In the proof of Theorem \ref{FV_identifiability_theorem} we observe that the Lebesgue-Stieltjes measures associated with $\hat{F}_n$ and $V(\blank;\hat{F}_n,\hat{m}_n) = V(\blank;\hat{F}_n,\smash{m_{\hat{F}_n}}) = \smash{\hat{F}_n^V(z)}$ are mutually absolutely continuous. As a consequence, $\hat{F}_n$ is necessarily also piecewise constant with the same jump locations as $\smash{\hat{F}_n^V}$. Therefore, we simply need to specify the value of $\hat{F}_n$ at $h_1,\dots,h_k$. Taking $z=h_1$ in (\ref{inverse_estimator_F_eq}), and using the fact that $\smash{\smallint_0^{h_1}\hat{F}_n(t)\mathrm{d}t} = 0$ we can solve for $\hat{F}_n(h_1)$:
\begin{equation}
    \hat{F}_n(h_1) = \frac{\hat{F}_n^V(h_1)}{\pi\left(\hat{m}_n - h_1\right)}. \label{inverse_estimator_F_first_value}
\end{equation}
In Appendix \ref{appendix_mn} an explicit formula for $\hat{m}_n$ is given, which also shows that $\hat{m}_n > h_1$. Let $i \in \{2,\dots,k\}$ then, from the proof of Theorem \ref{weighted_height_dist_thm} it follows that for the $\hat{F}_n$ we are looking for:
\begin{align*}
    \hat{F}_n^V(h_i) &= \hat{F}_n^V(h_{i-1}) + \int_{h_{i-1}}^{h_i}\pi \int_h^\infty \exp\left(-\pi \int_0^{u} \hat{F}_n(t)\mathrm{d}t \right)\mathrm{d}u\mathrm{d}\hat{F}_n(h)\\
    &= \hat{F}_n^V(h_{i-1}) + \pi \int_{h_{i}}^\infty \exp\left(-\pi \int_0^{u} \hat{F}_n(t)\mathrm{d}t \right)\mathrm{d}u \left( \hat{F}_n(h_i) -  \hat{F}_n(h_{i-1}) \right)
\end{align*}
Hence,
\begin{equation}
    \pi \int_{h_{i}}^\infty \exp\left(-\pi \int_0^{u} \hat{F}_n(t)\mathrm{d}t \right)\mathrm{d}u = \frac{\hat{F}_n^V(h_i) - \hat{F}_n^V(h_{i-1})}{\hat{F}_n(h_i) - \hat{F}_n(h_{i-1})}.\label{F_inverse_estimator_formula_derivation}
\end{equation}
Equation (\ref{inverse_estimator_induced_FV}) may be used to obtain an expression for $\hat{F}_n^V(h_i)$, plugging (\ref{F_inverse_estimator_formula_derivation}) into this expression and solving for $\hat{F}_n(h_i)$ yields:
\begin{align}
    \hat{F}_n(h_i) &= \hat{F}_n(h_{i-1})\left(\frac{\hat{F}_n^V(h_{i}) - 1 + \exp\left(-\pi \int_0^{h_{i}} \hat{F}_n(t)\mathrm{d}t\right)}{\hat{F}_n^V(h_{i-1}) - 1 + \exp\left(-\pi \int_0^{h_{i}} \hat{F}_n(t)\mathrm{d}t\right)}\right) \nonumber \\
    &= \hat{F}_n(h_{i-1})\left(\frac{\hat{F}_n^V(h_{i}) - 1 + \exp\left(-\pi \sum_{j=1}^{i-1}(h_i - h_j)\left(\hat{F}_n(h_j) - \hat{F}_n(h_{j-1}) \right) \right)}{\hat{F}_n^V(h_{i-1}) - 1 + \exp\left(-\pi \sum_{j=1}^{i-1}(h_i - h_j)\left(\hat{F}_n(h_j) - \hat{F}_n(h_{j-1}) \right) \right)}\right). \label{inverse_estimator_F_computational_formula}
\end{align}
Note that the RHS of (\ref{inverse_estimator_F_computational_formula}) only depends on the values $\hat{F}_n(h_j)$ with $j < i$. Hence, (\ref{inverse_estimator_F_first_value}) along with (\ref{inverse_estimator_F_computational_formula}) completely defines $\hat{F}_n$. From (\ref{inverse_estimator_F_computational_formula}) it is evident that $\hat{F}_n \in \mathcal{F}_{+}$, and this expression may be used to compute $\hat{F}_n$ in practice.

\subsection{Consistency}
In this section we show that $\hat{F}_n$, as in Definition \ref{inverse_estimator_F_def}, is a strongly consistent estimator for $F$. We start with a Lemma which implies that $\hat{m}_n$ is a strongly consistent estimator for $m_F$.

\begin{lemma}\label{mf_continuity_lemma}
    Let $(F_n)_{n \geq 1}$ be a sequence in $\mathcal{F}_{+}$, and let $F \in \mathcal{F}_{+}$ be non-zero. If $\lim_{n \to \infty} F_n(z) = F(z)$ for all $z\geq 0$, then $\lim_{n \to \infty}m_{F_n} = m_F$.
\end{lemma}

Next, we show that $\tilde{F}_n^V$ and $\hat{F}_n^V$ are strongly consistent and uniformly strongly consistent estimators of $F^V$ respectively.

\begin{lemma}\label{FV_estimators_consistency_lemma}
    With probability one, $\lim_{n \to \infty} \tilde{F}_n^V(z) = F^V(z)$ for all $z \geq 0$. Additionally, with probability one we have $\lim_{n \to \infty} \Vert \hat{F}_n^V - F^V \Vert_{\infty} = 0$. Here, $\tilde{F}_n^V$ and $\hat{F}_n^V$ are given by (\ref{Fv_estimator_def1}) and (\ref{Fv_estimator_def2}) respectively.
\end{lemma}
\begin{proof}
    We first show that with probability one, $\lim_{n \to \infty} \tilde{F}_n^V(z) = F^V(z)$ for all $z \geq 0$. Let $z \geq 0$, by Lemma \ref{monotone_estimator_lemma} it is sufficient to show that $\lim_{n \to \infty} \tilde{F}_n^V(z) = F^V(z)$ almost surely. Again, we apply the spatial ergodic theorem. This can be done since for all $(x,h) \in \eta$ we have: $C((x,h),\eta) - x = C((0,h),S_x\eta)$. Hence, by the translation invariance of Lebesgue measure, $\tilde{F}_n^V(z)$ may be written as:
    \[\tilde{F}_n^V(z) = \frac{1}{\nu_d(W_n)}\sum_{(x,h) \in \eta} \mathds{1}_{W_n}(x)\mathds{1}_{(0,z]}(h)\nu_d\left(C((0,h),S_x\eta)\right).\]
    So indeed, the spatial ergodic theorem yields $\lim_{n \to \infty} \tilde{F}_n^V(z) = F^V(z)$ almost surely. Similarly, we may argue that $\lim_{n \to \infty} \tilde{F}_n^V(\infty) = 1$ almost surely. Since $\hat{F}_n^V(z) = \tilde{F}_n^V(z)/\tilde{F}_n^V(\infty)$, we obtain via the continuous mapping theorem that $\lim_{n \to \infty} \hat{F}_n^V(z) = F^V(z)$ almost surely. The uniform strong consistency follows from repeating the steps in the proof of the Glivenko-Cantelli theorem.
\end{proof}

We need one more lemma before we prove the consistency result for $\hat{F}_n$.

\begin{lemma}\label{FV_continuity_lemma}
    Let $(F_n)_{n \geq 1}$ be a sequence in $\mathcal{F}_{+}$, and let $F \in \mathcal{F}_{+}$. Let $(m_n)_{n\geq 1}$ be a sequence in $(0,\infty)$ and let $m > 0$. If $\lim_{n \to \infty} F_n(z) = F(z)$ for all $z\geq 0$ and $\lim_{n \to \infty} m_n = m$, then $\lim_{n \to \infty}V(z;F_n,m_n) = V(z;F,m)$ for all $z \geq 0$.
\end{lemma}

\begin{theorem}[Consistency of $\hat{F}_n$]
    With probability one, $\lim_{n \to \infty}\hat{F}_n(z) = F(z)$ for all $z \geq 0$.
\end{theorem}

\begin{proof}
Let $(\Omega, \mathcal{A},\PP)$ be a probability space supporting a Poisson process $\eta$, with intensity measure $\nu_2 \times \mathbb{F}$. By Lemma \ref{mf_continuity_lemma} and Lemma \ref{FV_estimators_consistency_lemma} there exists a set $\Omega_0 \in \mathcal{A}$ with $\PP(\Omega_0)=1$ such that for all $\omega \in \Omega_0$ and $z \geq 0$ we have $\lim_{n \to \infty}\hat{F}_n^V(z;\omega) = V(z;F,m_F)$ and $\lim_{n \to \infty} \hat{m}_n(\omega) = m_F$. Let $z \geq 0$, we show that $\lim_{n \to \infty}\hat{F}_n(z;\omega) = F(z)$.

Pick $M > 0$ such that $F(z) < M$. For $n \in \NN$ and $h \geq 0$, define: $\bar{F}_n(h) = \min\{\hat{F}_n(h;\omega), M\}$. Then, $(\bar{F}_n)_{n \geq 1}$ is a uniformly bounded sequence of monotone functions. Let $(n_l)_{l \geq 1} \subset (n)_{n\geq 1}$ be an arbitrary subsequence. By Helly's selection principle there exists a further subsequence $(n_k)_{k\geq 1} \subset (n_l)_{l \geq 1}$ such that $\bar{F}_{n_k}$ converges pointwise to some monotone function $\bar{F}$ as $k \to \infty$. This implies that $\lim_{k \to \infty} \hat{F}_{n_k}(h;\omega) = \lim_{k \to \infty}\bar{F}_{n_k}(h) = \bar{F}(h)$ for all $h \in [0,R)$ with $R:= \sup\{h\geq 0: \bar{F}(h) < M\}$. By Lemma \ref{FV_continuity_lemma} we obtain along this subsequence: 
\[\lim_{k \to \infty}\hat{F}_{n_k}^V(h) := \lim_{k \to \infty}V(h; \hat{F}_{n_k}(\blank,\omega), \hat{m}_{n_k}(\omega)) = V(h;\bar{F},m_F) \text{ \ for all } h \in [0,R).\]
Because the whole sequence $\hat{F}_n^V(h;\omega)$ converges to $V(h;F,m_F)$ as $n \to \infty$, for $h \geq 0$, we obtain $V(h;F,m_F) = V(h;\bar{F},m_F)$ for all $h \in [0,R)$. Theorem \ref{FV_identifiability_theorem} now yields $F(h) = \bar{F}(h)$ for all $h \in [0,R)$, since $z \in [0,R)$, we have in particular $F(z) = \bar{F}(z)$. As a consequence: $\lim_{k \to \infty}\hat{F}_{n_k}(z;\omega) = F(z)$. Because the initial subsequence was chosen arbitrarily, the whole sequence converges: $\lim_{n \to \infty}\allowbreak \hat{F}_n(z;\omega) = F(z)$.
\end{proof}

\section{Stereology}\label{section_stereology}
In this section we study a special type of Poisson-Laguerre tessellations, namely sectional Poisson-Laguerre tessellations. By this we mean that we intersect a Poisson-Laguerre tessellation with a hyperplane, and we consider the resulting tessellation in this hyperplane. In Theorem 4.1. in \cite{Gusakova2024_2} it was shown that that intersecting a Poisson-Laguerre tessellation in $\RR^d$ with a hyperplane, yields a tessellation in this hyperplane which is again a Poisson-Laguerre tessellation. Because our parameterization is subtly different to the setting in \cite{Gusakova2024_2} we derive the intensity measure of the Poisson process corresponding to the sectional Poisson-Laguerre tessellation, for which we also use a different argument. 
%By identifying the intersecting hyperplane with $\RR^{d-1}$ we consider the following setting. 

Suppose we observe the extreme points and the corresponding cells, of a Poisson-Laguerre tessellation $L(\eta)$ in $\RR^{d-1}$, through the observation window $W_n$. The underlying Poisson process $\eta$ has intensity measure $\nu_{d-1} \times \mathbb{F}$, where $\mathbb{F}$ is a locally finite measure concentrated on $(0,\infty)$. Hence, we may use any of the estimators in the previous two sections to estimate $F(z) = \mathbb{F}((0,z])$. Throughout this section, we assume that $\bar{F}_n$ is a piecewise constant, strongly consistent estimator for $F$. Now, this Poisson-Laguerre tessellation in $\RR^{d-1}$ is the sectional tessellation corresponding to a Poisson-Laguerre tessellation $L(\Psi)$ in $\RR^d$. The Poisson process $\Psi$ of this higher-dimensional tessellation has intensity measure $\nu_d \times \mathbb{H}$, where $\mathbb{H}$ is a locally finite measure concentrated on $(0,\infty)$. For $z \geq 0$ define: $H(z) := \mathbb{H}((0,z])$. In this section, we show how $F$ is related to $H$, and how a consistent estimator for $F$ can be used to obtain a (locally) consistent estimator for $H$. Thereby, we have a solution to the stereological problem. First, we need the following lemma for obtaining an expression for $F$ in this stereological setting:

\begin{lemma}\label{laguerre_sectional_lemma}
    Let $\varphi \subset \RR^d \times (0,\infty)$ be an at most countable set. Let $\theta \in \mathbb{S}^{d-1}$ and $s \in \RR$. Define the hyperplane $T := \{y \in \RR^d: \langle \theta, y \rangle = s\}$. For $(x,h) \in \varphi$, with $x \in \RR^d$ and $h > 0$ let:
    \begin{align*}
        x' &:= x - (\langle \theta, x \rangle - s)\theta\\
        h' &:= h + \Vert x' - x \Vert^2 =  h + (\langle \theta, x\rangle - s)^2.
    \end{align*}
    Note that $x' \in T$ and define $\varphi' := \{(x',h'):(x,h) \in \varphi\}$. Then, for all $(x,h) \in \varphi$: 
    $C((x,h),\varphi) \cap T = C'((x',h'),\varphi')$ with:
    \[C'((x',h'), \varphi') = \{y \in T: \Vert y-x' \Vert^2 + h' \leq \Vert y-\bar{x} \Vert^2 + \bar{h} \text{ for all } (\bar{x},\bar{h}) \in \varphi')\}. \]
\end{lemma}

\begin{proof}
    Let $y \in T$ and $(x,h) \in \varphi$, then a direct computation yields:
\[||x' - y||^2 = ||x - y||^2 - 2(\langle \theta, x\rangle - s)\langle x - y, \theta\rangle + (\langle \theta, x\rangle - s)^2 = ||x - y||^2 - h' + h.\]   
Since $||x' - y||^2 + h'  = ||x - y||^2 + h$, the claim follows.
\end{proof}

This lemma describes the set of weighted points which generates a sectional Laguerre diagram. We now apply this to the Poisson-Laguerre tessellation generated by the Poisson process $\Psi$. Because a Poisson-Laguerre tessellation is stationary and isotropic the choice of hyperplane does not affect the distribution of the sectional tessellation. For $x \in \RR^d$ write: $x = (\mathrm{x}_1, \mathrm{x}_2, \dots, \mathrm{x}_d)$. We choose the hyperplane $\mathrm{x}_{d} = 0$ which corresponds to taking $\theta = (0,\dots,0,1) \in  \mathbb{S}^{d-1}$ and $s=0$ in Lemma \ref{laguerre_sectional_lemma}. In view of Lemma \ref{laguerre_sectional_lemma} consider the function which maps a pair $(x,h) \in \Psi$ to the corresponding $(x', h')$. Hence, this function is given by $(\mathrm{x}_1,\dots,\mathrm{x}_d,h) \mapsto (\mathrm{x}_1,\dots,\mathrm{x}_{d-1},0,h+\mathrm{x}_d^2)$. Naturally, the $d$-th component of the resulting vector is always zero. We identify the hyperplane $\mathrm{x}_{d} = 0$ with $\RR^{d-1}$ and therefore we consider the function $\tau:\RR^d \times (0,\infty) \to \RR^{d-1} \times (0,\infty)$ which is defined via: $\tau(x,h) = (\mathrm{x}_1,\dots,\mathrm{x}_{d-1}, h + \mathrm{x}_{d}^2)$. Hence, the point process $\eta := \tau(\Psi)$ generates the sectional tessellation. By the mapping theorem (see Theorem 5.1 in \cite{Last2018}) $\eta$ is again a Poisson process on $\RR^{d-1} \times (0,\infty)$ with intensity measure: $\EE(\Psi(\tau^{-1}(\blank)))$. Let $B \subset \RR^{d-1}$ be a Borel set and let $z \geq 0$. Note that $h+\mathrm{x}_d^2 \leq z$ if and only if $h \leq z$ and $\mathrm{x}_d \in [-\sqrt{z-h},\sqrt{z-h}]$. As a result:
\begin{equation}
    \tau(x,h) \in B \times(0,z] \iff x \in B \times \left[-\sqrt{z-h},\sqrt{z-h}\right] \text{ and } h \leq z.\label{sectional_derivation_eq}
\end{equation}
Via the Campbell formula and (\ref{sectional_derivation_eq}) we find:
\begin{align*}
    \EE\left(\Psi(\tau^{-1}(B\times(0,z]))\right) &= \int_{\RR^d}\int_0^\infty \mathds{1}{\left\{\tau(x,h) \in B \times (0,z]\right\}}\mathrm{d}H(h)\mathrm{d}x \\
    &= \int_{\RR^d}\int_0^z \mathds{1}{\left\{x \in B \times \left[-\sqrt{z-h},\sqrt{z-h}\right]\right\}}\mathrm{d}H(h)\mathrm{d}x \\
    &= \nu_{d-1}(B)2\int_0^z \sqrt{z-h}\mathrm{d}H(h). 
\end{align*}
Hence, we obtain:
\[F(z) = 2\int_0^z \sqrt{z-h}\mathrm{d}H(h).\]
Let us discuss some properties of this function $F$. First of all, $F$ is not a bounded function. Indeed, choose $z_0 > 0$ such that $H(z_0) > 0$, and let $z > z_0$, via integration by parts we observe:
\begin{equation}
F(z) = \int_0^z H(h)\frac{1}{\sqrt{z-h}}\mathrm{d}h \geq \int_{z_0}^z H(h)\frac{1}{\sqrt{z}}\mathrm{d}h \geq  H(z_0)\frac{z - z_0}{\sqrt{z}}.\label{F_unbounded_eq}
\end{equation}
It immediately follows that $\lim_{z \to \infty}F(z) = \infty$. Another property of $F$ is that it is absolutely continuous, and has a Lebesgue density $f$ given by:
\[f(z) = \int_0^z \frac{1}{\sqrt{z-t}}\mathrm{d}H(t).\]

It is possible to express $H$ in terms of $F$, because this is an Abel integral equation. For a direct derivation of the inversion formula see for example \cite{Srivastav1963}. Here, we simply show that the following expression is indeed an inversion formula for $H(z)$:
\begin{align}
    \frac{1}{\pi}\int_0^z \frac{1}{\sqrt{z-t}}\mathrm{d}F(t) &= \frac{1}{\pi} \int_0^z\frac{1}{\sqrt{z-t}}\int_0^t \frac{1}{\sqrt{t-s}}\mathrm{d}H(s)\mathrm{d}t \nonumber \\
    &= \int_0^z \int_s^z \frac{1}{\sqrt{z-t}\sqrt{t-s}}\mathrm{d}t\mathrm{d}H(s) \nonumber\\
    &= \int_0^z \int_0^1 \frac{1}{\pi}(1-u)^{-\frac{1}{2}}u^{-\frac{1}{2}}\mathrm{d}u\mathrm{d}H(s) \label{stereology_inversion_eq1}\\
    &= H(z) \label{stereology_inversion_eq2}.
\end{align}
In (\ref{stereology_inversion_eq1}) we substituted $u = (t-s)/(z-s)$. Finally, (\ref{stereology_inversion_eq2}) follows from the fact that the inner integral in (\ref{stereology_inversion_eq1}) is equal to one, since this integral represents the Beta function evaluated in $(1/2,1/2)$. A plugin estimator for $H(z)$ is therefore given by:
\begin{equation}
    H_n(z) := \frac{1}{\pi}\int_0^z \frac{1}{\sqrt{z-t}}\mathrm{d}\bar{F}_n(t),
\end{equation}
where $\bar{F}_n$ is a piecewise constant, strongly consistent estimator for $F$. This estimator is however rather ill-behaved. While $H$ is a monotone function, $H_n$ is not. Because $\bar{F}_n$ is piecewise constant, $H_n$ is decreasing between jump locations of $\bar{F}_n$. Moreover, if $z_0$ is a jump location of $\bar{F}_n$, then $\lim_{z\downarrow z_0}H_n(z) = \infty$. Therefore, we use isotonization to obtain an estimator for $H$ which is monotone, and show that it is consistent. We note that our estimator is similarly defined as the isotonic estimator in \cite{Groeneboom1995}. For the remainder of this section, let $k = k(n)$ be the number of jump locations of $\bar{F}_n$. Let $h_1,h_2,\dots,h_{k}$ with $0 < h_1 < h_2 <\dots< h_{k} < \infty$ be the jump locations of $\bar{F}_n$. In order to introduce the isotonic estimator we define for $z \geq 0$:
\begin{equation}
    U_n(z) := \int_0^z H_n(t)\mathrm{d}t = \frac{2}{\pi}\int_0^z \sqrt{z-t}\mathrm{d}\bar{F}_n(t). \label{primitive_plugin_estimator}
\end{equation}
Choose (a large) $M > 0$ and write $z_M := \min\{h_k, M\}$. Let $U_n^M$ be the greatest convex minorant of $U_n$ on $[0,z_M]$. That is, $U_n^M$ is the greatest convex function on $[0,z_M]$ which lies below $U_n$. Then, define:
\begin{equation}
    \hat{H}_n^M(z) := \begin{cases}\label{local_estimator_H}
        U_n^{M,r}(z) & \text{ if } z \in \bigl[0,z_M\bigr)\\
        U_n^{M,l}(z_M) & \text{ if } z \geq z_M,
    \end{cases}
\end{equation}
where $U_n^{M,l}$, $U_n^{M,r}$ denote the left- and right-derivative of $U_n^M$ respectively. The reason we cannot simply extend the definition of $U_n^M$ to the whole of $[0,\infty)$ is due to the fact that $U_n$ is concave on $[h_{k}, \infty)$. As a result, the greatest convex minorant of $U_n$ on $[0,\infty)$ is the zero function. Because of the convexity of $U_n^M$ on $[0,z_M]$, $\hat{H}_n^M$ is guaranteed to be non-decreasing, and is referred to as an isotonic estimator. Analogously to (\ref{primitive_plugin_estimator}) we define for $z \geq 0$:
\begin{equation}
    U(z) := \int_0^z H(t)\mathrm{d}t = \frac{2}{\pi}\int_0^z \sqrt{z-t}\mathrm{d}F(t). \label{primitive_F_eq}
\end{equation}
Note that $U^r(z) = H(z)$, so indeed, the right-derivative of $U_n^M$ is a natural choice for an estimator of $H$. In the next theorem we prove consistency of $\hat{H}_n^M$. Currently, it is not known whether $\hat{H}_n := \hat{H}_n^\infty$ is a globally consistent estimator.

\begin{theorem}[Consistency of $\hat{H}_n^M$]
Let $M > 0$ and let $\hat{H}_n^M$ be as in (\ref{local_estimator_H}). Let $z \in [0,M)$, then with probability one:
    \[H(z-) \leq \liminf_{n \to \infty} \hat{H}_n^M(z) \leq \limsup_{n \to \infty} \hat{H}_n^M(z) \leq H(z).\]
    In particular, if $z$ is a continuity point of $H$: $\lim_{n \to \infty} \hat{H}_n^M(z) = H(z)$ almost surely.
\end{theorem}

\begin{proof}
    Let $z \in [0,M)$. Because $\bar{F}_n$ is piecewise constant and a consistent estimator of the unbounded function $F$ (recall equation (\ref{F_unbounded_eq})), it follows that $\lim_{n \to \infty} h_{k(n)} = \infty$ almost surely. Let $(\Omega, \mathcal{A},\PP)$ be a probability space supporting a Poisson process $\eta$, with intensity measure $\nu_{d-1} \times \mathbb{F}$. Choose $\Omega_0 \in \mathcal{A}$ with $\PP(\Omega_0)=1$ such that for all $\omega \in \Omega_0$ we have $\lim_{n \to \infty} h_{k(n)}(\omega) = \infty$ and $\lim_{n \to \infty}\bar{F}_n(h;\omega) = F(h)$ for all $h \geq 0$. For the remainder of the proof, take $n$ sufficiently large such that $h_{k(n)}(\omega) > M$. Note how $U_n$ and $U$ depend on $\bar{F}_n$ and $F$ respectively, see (\ref{primitive_plugin_estimator}) and (\ref{primitive_F_eq}). As a consequence, the pointwise convergence of $\bar{F}_n(\blank;\omega)$ to $F$ implies: $\lim_{n \to \infty} U_n(x;\omega) = U(x)$ for all $x \geq 0$. Note that $U$ is non-decreasing and continuous, therefore the convergence is also uniform on $[0,M]$. That is, $\lim_{n \to \infty}\sup_{x \in [0,M]}\left|U_n(x;\omega) - U(x)\right| = 0$. Because $U$ is defined as the integral of a non-decreasing function, it is convex. A variant of Marshall's lemma (the convex analogue of 7.2.3. on p. 329 in \cite{Robertson1988}) directly yields:
    \[\sup_{x \in [0,M]}\left|U_n^M(x;\omega) - U(x)\right| \leq \sup_{x \in [0,M]}\left|U_n(x;\omega) - U(x)\right|.\]
    Therefore, we also have $\lim_{n \to \infty}\sup_{x \in [0,M]}\left|U_n^M(x;\omega) - U(x)\right| = 0$. Take $\delta > 0$ such that $z + \delta < M$. Then, for each $0 < h < \delta$ we have by the convexity of $U_n^M$:
    \[\frac{U_n^M(z;\omega)- U_n^M(z-h;\omega)}{h} \leq U_n^{M,l}(z;\omega) \leq U_n^{M,r}(z;\omega) \leq \frac{U_n^M(z;\omega)- U_n^M(z+h;\omega)}{h}.\]
    By using $\lim_{n \to \infty}\sup_{x \in [0,M]}|U_n^M(x;\omega) - U(x)| = 0$, the following holds:
    \[\frac{U(z)- U(z-h)}{h} \leq \liminf_{n \to \infty}U_n^{M,r}(z;\omega) \leq \limsup_{n \to \infty}  U_n^{M,r}(z;\omega) \leq \frac{U(z)- U(z+h)}{h}.\]
    The result follows from letting $h \downarrow 0$ and by recognizing that $U^l(z) = H(z-)$ and $U^r(z) = H(z)$. 
\end{proof}

\begin{remark}
    By choosing $M > 0$ very large, the estimators $\hat{H}_n := \hat{H}_n^\infty$ and $\hat{H}_n^M$ will in practice often coincide, since we will typically observe $h_k < M$. Therefore, in the remainder of this paper we will only consider computational aspects and simulation performance of the estimator $\hat{H}_n$. 
\end{remark}

% It can be shown that $\hat{H}_n$ is piecewise constant, and can only have jumps at the locations $h_1,h_2,\dots,\allowbreak h_k$. Let $y_i = (U_n(h_{i+1}) - U_n(h_{i}))/(h_{i+1} - h_{i})$, and $w_i = h_{i+1} - h_{i}$. Then, by setting $\hat{\beta} = (\hat{H}_n(h_1),\allowbreak\hat{H}_n(h_2),\dots,\hat{H}_n(h_{k-1}))$, we can compute $\hat{H}_n$ via:
% \begin{equation}
%     \hat{\beta} = \argmin_{\beta \in \mathcal{C}_{+}} \sum_{i=1}^{k-1} (\beta_i - y_i)^2w_i , \label{isotonic_regression_characterization}
% \end{equation}
% where the closed convex cone $\mathcal{C}_{+}$ is given by: $\mathcal{C}_{+} := \{\beta \in \RR^{k-1}: 0\leq \beta_1 \leq \beta_2 \leq \dots \leq \beta_{k-1}\}$. Finally, we have: $\hat{H}_n(h_{k}) = \hat{H}_n(h_{k-1})$. The optimization problem in (\ref{isotonic_regression_characterization}) is an isotonic regression problem. We note that implementations for solving this problem are widely available. A derivation of (\ref{isotonic_regression_characterization}) is given in the Supplementary Material \cite{vdjagt2025}. In Figure \ref{fig:stereology_example} a realization of $H_n$ and the corresponding realization of $\hat{H}_n$ is shown.

%\section{Computation of the isotonic estimator}\label{isotonic_section}
We now show that computing the isotonic estimator $\hat{H}_n$ is equivalent to solving an isotonic regression problem. This is achieved via the following lemma, which is a straightforward modification of Lemma 2 in \cite{Groeneboom1995}. 

 \begin{lemma}\label{L2_projection_lemma}
     Let $M>0$, and let $\varphi$ be an a.e. continuous non-negative function on $[0,M]$. Define the function $\Phi$, for $z \geq 0$ as:
     \[\Phi(z) = \int_0^z \varphi(x)\mathrm{d}x.\]
     Let $\Phi^*$ be the greatest convex minorant of $\Phi$ on $[0,M]$. Let $\Phi^{*,r}$ be the right-derivative of $\Phi^*$, then:
     \[\int_0^M (\varphi(x) - \psi(x))^2\mathrm{d}x \geq \int_0^M (\varphi(x) - \Phi^{*,r}(x))^2\mathrm{d}x + \int_0^M (\Phi^{*,r} - \psi(x))^2\mathrm{d}x,\]
     for all functions $\psi$ in the set:
     \[\mathcal{F}_M := \left\{\psi:[0,M]\to[0,\infty): \psi \text{ is non-decreasing and right-continuous}\right\}.\]
     %$\in \mathcal{F}_{+}$ satisfying $\psi(M) \leq \Phi^{*,r}(M)$. 
 \end{lemma}

We use Lemma \ref{L2_projection_lemma} to show that $\hat{H}_n$ may be interpreted as the $L^2$-projection of $H_n$ on the space of monotone functions. Recall that $h_1,h_2,\dots,h_{k}$ are the unique jump locations of $\bar{F}_n$. Additionally, let $h_0 = 0$. Define $\tilde{H}_n$ to be the piece-wise constant function on $[0,h_k)$ which is given by:
\[\tilde{H}_n(z) = \frac{U_n(h_{i+1}) - U_n(h_{i})}{h_{i+1} - h_{i}}, \text{ \quad } z \in [h_{i},h_{i+1}), \text{ \ } i \in \{0,1,\dots,k-1\}.\]
For $z \in [0,h_{k}]$, let: $\tilde{U}_n(z) = \int_0^z \tilde{H}_n(t)\mathrm{d}t$. Then, $U_n(h_i) = \tilde{U}_n(h_i)$ for all $i \in \{0,1,\dots,k\}$. While $U_n$ is concave between successive jump locations (due to the square root), $\tilde{U}_n$ is linear between successive jump locations. As a consequence, $U_n$ and $\tilde{U}_n$ have the same greatest convex minorant. Hence, $\hat{H}_n(z) = U_n^{*,r}(z) = \tilde{U}_n^{*,r}(z)$, for $z \in [0,h_{k})$. Finally, by taking $\varphi = \tilde{H}_n$ (and $\varphi = H_n$) and $M = h_{k}$ in Lemma \ref{L2_projection_lemma} we see that:
\begin{equation}
     U_n^{*,r} = \argmin_{H \in \mathcal{F}_{h_{k}}}\int_0^{h_{k}} (H(x) - H_n(x))^2 \mathrm{d}x = \argmin_{H \in \mathcal{F}_{h_{k}}}\int_0^{h_{k}} (H(x) - \tilde{H}_n(x))^2 \mathrm{d}x.\label{minimization_equation}
 \end{equation}

\begin{figure}[t!]
    \centering
    \includegraphics{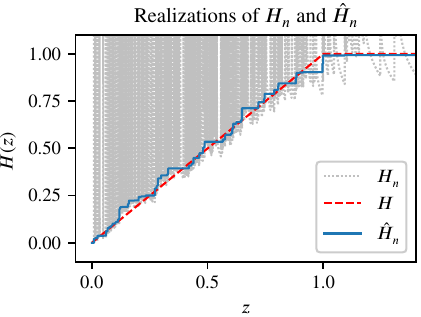}
    \caption{A comparison of the plugin estimator $H_n$ and the isotonic estimator $\hat{H}_n$. The actual underlying $H$ is equal to the CDF of a uniform distribution on $(0, 1)$. }
    \label{fig:stereology_example}
\end{figure}

Because $\tilde{H}_n$ is piece-wise constant, in (\ref{minimization_equation}) we may even minimize over all functions in $\mathcal{F}_{h_k}$ which are also piece-wise constant with jump locations at $h_1,h_2,\dots,h_k$. Hence, $\hat{H}_n$ is piece-wise constant and when solving the minimization problem in (\ref{minimization_equation}) we only seek to determine the values $\hat{H}_n$ attains at these jump locations. Let $y_i = \tilde{H}_n(h_i)$, and $w_i = h_{i+1} - h_{i}$. Then, by setting $\hat{\beta} = (\hat{H}_n(h_1),\hat{H}_n(h_2),\dots,\hat{H}_n(h_{k-1}))$, (\ref{minimization_equation}) may be written as:
 \begin{equation}
     \hat{\beta} = \argmin_{\beta \in \mathcal{C}_{+}} \sum_{i=1}^{k-1} (\beta_i - y_i)^2w_i , \label{isotonic_regression_characterization}
 \end{equation}
 where the closed convex cone $\mathcal{C}_{+}$ is given by: $\mathcal{C}_{+} := \{\beta \in \RR^{k-1}: 0\leq \beta_1 \leq \beta_2 \leq \dots \leq \beta_{k-1}\}$. Finally, observe that $\hat{H}_n(h_{k}) = \hat{H}_n(h_{k-1})$. The optimization problem in (\ref{isotonic_regression_characterization}) is indeed an isotonic regression problem. We note that implementations for solving this problem are widely available. In Figure \ref{fig:stereology_example} a realization of $H_n$ and the corresponding realization of $\hat{H}_n$ is shown.

\begin{figure}[t!]
\centering
\makebox[\linewidth]{\makebox[\linewidth]{
    \begin{subfigure}[t]{0.5\linewidth}
        \centering
        \includegraphics{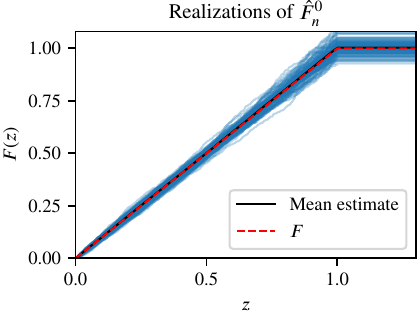}
        %\caption{}
    \end{subfigure}
    \begin{subfigure}[t]{0.5\linewidth}
        \centering
        \includegraphics{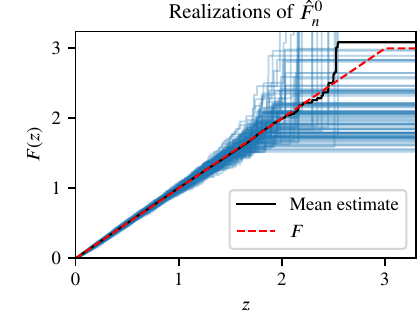}
        %\caption{}
    \end{subfigure}\hfill}}
    \makebox[\linewidth]{\makebox[\linewidth]{
    \begin{subfigure}[t]{0.5\linewidth}
        \centering
        \includegraphics{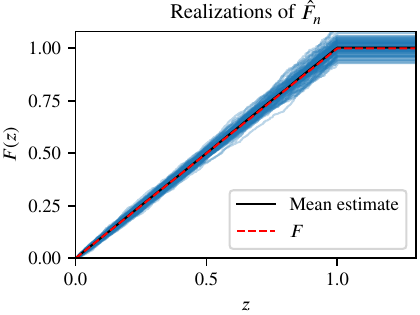}
        %\caption{}
    \end{subfigure}
    \begin{subfigure}[t]{0.5\linewidth}
        \centering
       \includegraphics{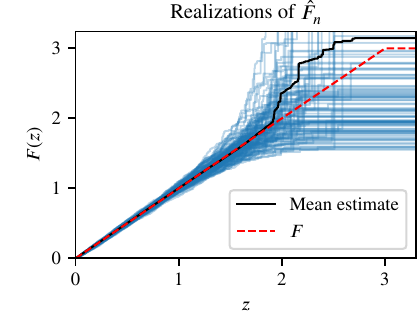}
        %\caption{}
    \end{subfigure}\hfill}}
    \caption{Simulation results for $\hat{F}_n^0$ and $\hat{F}_n$, where $F$ is given by (\ref{true_F_option1}), with $M=1$ (Left) and $M=3$ (Right).}\label{Figure_uniform_simulations}
\end{figure}

\section{Simulations}\label{section_simulations}

In previous sections we have derived consistent estimators for the distribution function corresponding to the underlying Poisson process $\eta$, both in the direct setting ($\RR^d$) and in the stereological setting ($\RR^{d-1}$). Additionally, we have shown how to compute these estimators. In this section we perform some simulations such that we can assess their performance. We compute Laguerre tessellations using the Voro++ software \cite{Rycroft2009}. For the estimators $\hat{F}_n^{0}$ and $\hat{F}_n$ we focus on the case $d=2$. Let $M > 0$ and $z\geq0$, we consider the following choices for the underlying $F$.
\begin{align}
    F_1(z;M) &= z\cdot\mathds{1}\{z < M\} + M\cdot\mathds{1}\{z \geq M\} \label{true_F_option1}\\
    F_2(z) &= 0.01\cdot\mathds{1}\{z \geq 1\} + 0.04\cdot\mathds{1}\{z \geq 8\} + 0.95\cdot\mathds{1}\{z \geq 10\}.\label{true_F_option2}
\end{align}

Note that $F_2$ corresponds to the $F$ in Example \ref{example_discrete_laguerre_text}. For both choices of $F$ it is simple to simulate a corresponding Poisson process, because these Poisson processes can be recognized as independently marked homogeneous Poisson processes. Throughout this section we write $P_n:=\EE(\eta^{0}(W_n \times (0,\infty))$. We choose a square observation window $W_n$ such that $P_n = 1000$. In words, we choose a square $W_n$ with an area such that the expected number of observed points of $\eta^0$ in $W_n$ is equal to 1000. First, we consider $F_1$ as the underlying truth. For $M=1$ and $M=3$ we repeat the simulation procedure with this $F$ 100 times, such that we obtain 100 realizations of $\hat{F}_n^0$ and $\hat{F}_n$ for each value of $M$. For each estimator, and each choice of $M$, we also compute the pointwise average of all realizations. The results are shown in Figure \ref{Figure_uniform_simulations}. A blue line is a realization of an estimator, a black line is a pointwise average. We can clearly see that estimates of $F(z)$ for $z$ close to zero are much more accurate than estimates of $F(z)$ for large values of $z$. This is especially evident for the results corresponding to $M=3$. This is not too surprising in view of the crystallization interpretation of a Laguerre tessellation as described in section \ref{section_introduction_pois_laguerre}. We expect that points with large weights are less likely to generate non-empty cells. As a result we sample points with large weights less often, which makes estimation of $F(z)$ for large values of $z$ more difficult. This also means that we expect that the accuracy of an estimate of $F$ near zero is much more important if we wish to use this estimate to simulate a Poisson-Laguerre tessellation which is similar to the observed tessellation. From Figure \ref{Figure_uniform_simulations} it is not very clear whether there are significant differences between $\hat{F}_n^0$ and $\hat{F}_n$, though it does seem that $\hat{F}_n^0$ performs slightly better on average when $z$ is large.

\begin{table}[t!]
\centering
\small
\begin{tabular}
    {
  r
  S[table-format=1.6]
>{{{\lp}}} % Add square bracket before column
S[round-precision=2, table-format = -1.4,table-space-text-pre=\lp]
@{,\,} % Add comma and thin-space between the columns
S[round-precision=2, table-format = 1.4,table-space-text-post=\rp]
<{{{\rp}}} % Add square bracket after column
    S[table-format=1.5]
>{{{\lp}}} % Add square bracket before column
S[round-precision=2,table-format = -1.4,table-space-text-pre=\lp]
@{,\,} % Add comma and thin-space between the columns
S[round-precision=2,table-format = 1.4,table-space-text-post=\rp]
<{{{\rp}}} % Add square bracket after column
 S[table-format=1.4]
>{{{\lp}}} % Add square bracket before column
S[round-precision=2, table-format = -1.3,table-space-text-pre=\lp]
@{,\,} % Add comma and thin-space between the columns
S[round-precision=2, table-format = 1.3,table-space-text-post=\rp]
<{{{\rp}}} % Add square bracket after column
@{}l@{}
}
\toprule
\multicolumn{1}{c}{} & \multicolumn{3}{c}{$F(1) - \hat{F}_n^{0}(1)$} & \multicolumn{3}{c}{$F(8) - \hat{F}_n^{0}(8)$} & \multicolumn{3}{c}{$F(10) - \hat{F}_n^{0}(10)$}&\\
\cmidrule(r){1-1}\cmidrule(lr){2-4}\cmidrule(lr){5-7}\cmidrule(l){8-11} 
  \multicolumn{1}{c}{$P_n$} &        {mean ($|\blank|$)} &    \multicolumn{2}{c}{(2.5\%, 97.5\%)} & {mean ($|\blank|$)} &    \multicolumn{2}{c}{(2.5\%, 97.5\%)} & {mean ($|\blank|$)} &    \multicolumn{2}{c}{(2.5\%, 97.5\%)} & \\
\cmidrule(r){1-1}\cmidrule(lr){2-4}\cmidrule(lr){5-7}\cmidrule(l){8-11} 
 500 & 0.002835 & -0.005655 & 0.005183 & 0.007339 & -0.018223 & 0.017858 & 0.042964 & -0.113546 & 0.100908 &   \\
1000 & 0.001966 & -0.004195 & 0.004581 & 0.004241 & -0.010065 & 0.009870 & 0.032670 & -0.065045 & 0.090793 &   \\
2000 & 0.001452 & -0.003405 & 0.002933 & 0.003488 & -0.007925 & 0.008054 & 0.019188 & -0.041667 & 0.046299 &   \\
5000 & 0.000845 & -0.001922 & 0.001995 & 0.002058 & -0.004617 & 0.005238 & 0.014588 & -0.030272 & 0.030966 &   \\
\bottomrule
\end{tabular}
\caption{Simulation results for $\hat{F}_n^0$, where $F$ is given by (\ref{true_F_option2}).}\label{table_estimator_F0}

\begin{tabular}
    {
  r
  S[table-format=1.6]
>{{{\lp}}} % Add square bracket before column
S[round-precision=2, table-format = -1.4,table-space-text-pre=\lp]
@{,\,} % Add comma and thin-space between the columns
S[round-precision=2, table-format = 1.4,table-space-text-post=\rp]
<{{{\rp}}} % Add square bracket after column
    S[table-format=1.5]
>{{{\lp}}} % Add square bracket before column
S[round-precision=2,table-format = -1.4,table-space-text-pre=\lp]
@{,\,} % Add comma and thin-space between the columns
S[round-precision=2,table-format = 1.4,table-space-text-post=\rp]
<{{{\rp}}} % Add square bracket after column
 S[table-format=1.4]
>{{{\lp}}} % Add square bracket before column
S[round-precision=2, table-format = -1.2,table-space-text-pre=\lp]
@{,\,} % Add comma and thin-space between the columns
S[round-precision=2, table-format = 1.2,table-space-text-post=\rp]
<{{{\rp}}} % Add square bracket after column
@{}l@{}
}
\toprule
\multicolumn{1}{c}{} & \multicolumn{3}{c}{$F(1) - \hat{F}_n^{0}(1)$} & \multicolumn{3}{c}{$F(8) - \hat{F}_n^{0}(8)$} & \multicolumn{3}{c}{$F(10) - \hat{F}_n^{0}(10)$}&\\
\cmidrule(r){1-1}\cmidrule(lr){2-4}\cmidrule(lr){5-7}\cmidrule(l){8-11} 
  \multicolumn{1}{c}{$P_n$} &        {mean ($|\blank|$)} &    \multicolumn{2}{c}{(2.5\%, 97.5\%)} & {mean ($|\blank|$)} &    \multicolumn{2}{c}{(2.5\%, 97.5\%)} & {mean ($|\blank|$)} &    \multicolumn{2}{c}{(2.5\%, 97.5\%)} & \\
\cmidrule(r){1-1}\cmidrule(lr){2-4}\cmidrule(lr){5-7}\cmidrule(l){8-11}  
 500 & 0.002939 & -0.006894 & 0.005743 & 0.007615 & -0.019025 & 0.016683 & 0.386360 & -1.911011 & 0.465305 &   \\
1000 & 0.001975 & -0.003904 & 0.005038 & 0.004419 & -0.010097 & 0.010227 & 0.267426 & -1.100581 & 0.332318 &   \\
2000 & 0.001510 & -0.003686 & 0.003057 & 0.003538 & -0.007421 & 0.008251 & 0.169637 & -0.608193 & 0.252847 &   \\
5000 & 0.000885 & -0.002042 & 0.002024 & 0.002134 & -0.004810 & 0.004501 & 0.107238 & -0.300698 & 0.210355 &   \\
\bottomrule
\end{tabular}
\caption{Simulation results for $\hat{F}_n$, where $F$ is given by (\ref{true_F_option2}).}\label{table_estimator_F}
\end{table}

\begin{figure}[t!]
\centering
    \makebox[\linewidth]{\makebox[\linewidth]{
    \begin{subfigure}[t]{0.5\linewidth}
        \centering
        \includegraphics{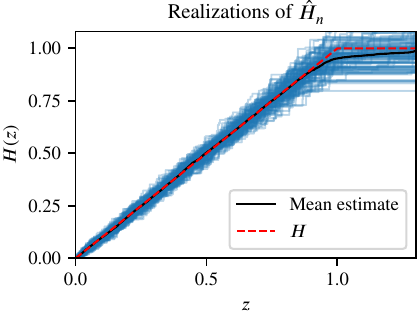}
        %\caption{}
    \end{subfigure}
    \begin{subfigure}[t]{0.5\linewidth}
        \centering
        \includegraphics{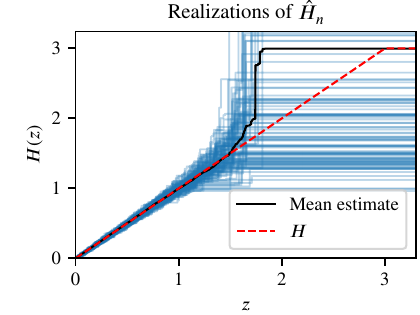}
        %\caption{}
    \end{subfigure}\hfill}}
    \caption{Simulation results for $\hat{H}_n$ where $H$ is given by (\ref{true_F_option1}), with $M=1$ (Left) and $M=3$ (Right).}\label{Figure_stereology_uniform_simulations}
\end{figure}

Now, we consider $F_2$ as the underlying truth. For this choice of $F$ we only observe points with weights in the set $\{1,8,10\}$. As a result, realizations of $\hat{F}_n^{0}$ and $\hat{F}_n$ will only have jumps at these values. We can therefore easily quantify the error of $\hat{F}_n^{0}$ by computing $F(z) - \hat{F}_n^{0}(z)$ for $z \in \{1,8,10\}$. Of course, we can do the same for $\hat{F}_n$. Again, we repeat the simulation procedure 100 times. This time however, we also repeat this for multiple choices of observation windows. We choose $W_n$ such that $P_n$ is equal to 500, 1000, 2000 or 5000. The simulation results are shown in Tables \ref{table_estimator_F0} and \ref{table_estimator_F}. This table contains the mean over all 100 absolute errors for each choice of $W_n$ and for each choice of $z$. We also include the $2.5\%$ and $97.5\%$ quantiles of these 100 errors. We can see that at $z=1$ and $z=8$ the performance of the estimators $\hat{F}_n^{0}$ and $\hat{F}_n$ is quite similar. However, at $z=10$ it is clear that $\hat{F}_n^{0}$ performs much better. This is somewhat surprising, after all, $\hat{F}_n$ takes into account more information than $\hat{F}_n^0$. We do not yet know whether the difference in performance is due to differences in numerical stability of the inversion procedures or due to different rates of convergence of the estimators. We also should point out that for a single realization of $\hat{F}_n$ corresponding to $P_n = 500$ we observed a numerical overflow. That is, we observed: $\hat{F}_n(10) <\hat{F}_n(8)$. It may therefore be of future interest to study whether there are more numerically stable ways to compute $\hat{F}_n$.

Finally, we show some simulation results for $\hat{H}_n$. In the previous simulations we observed that $\hat{F}_n^{0}$ performs better than $\hat{F}_n$. Therefore, we compute $\hat{H}_n$ via $\bar{F}_n = \hat{F}_n^{0}$. We consider $d=3$ such that we observe a 2D sectional tessellation. We take the underlying $H$ equal to $F_1$ as in (\ref{true_F_option1}). Again, we choose $W_n$ such that $P_n=1000$ and perform 100 repeated simulations for both $M=1$ and $M=3$. The results are shown in Figure \ref{Figure_stereology_uniform_simulations}. As expected, in the stereological setting we observe a bigger variance in realizations of $\hat{H}_n$ compared to the realizations shown in Figure \ref{Figure_uniform_simulations}. Overall, all estimators seem to perform satisfactory.

\appendix
\section{Proofs}\label{appendix_proofs}

\begin{proof}[Proof of Lemma \ref{monotone_estimator_lemma}]
    Let $(\Omega, \mathcal{A},\PP)$ be a probability space supporting the sequence $(F_n)_{n \geq 1}$. For $z \in \RR$, there exists by assumption a set $\Omega_z \in \mathcal{A}$ be such that $\lim_{n \to \infty}F_n(z;\omega) = F(z)$ for all $\omega \in \Omega_z$ and $\PP(\Omega_z) = 1$. Define $D := \{z \in \RR : z \in \mathbb{Q} \text{ or } z \text{ is a discontinuity point of $F$}\}$. Because monotone functions have at most countably many discontinuity points and because the rationals are countable it follows that $D$ is countable. Letting $\Omega' := \cap_{z \in D}\Omega_z$ we obtain $\PP(\Omega') = 1$. Let $z \in \RR$ and $\omega \in \Omega'$, we show that $\lim_{n \to \infty}F_n(z;\omega) = F(z)$. If $z \in \mathbb{Q}$ or if $z$ is a discontinuity point of $F$, then the result is immediate. Suppose that $z \in \RR \setminus \mathbb{Q}$ is a continuity point of $F$. For $m \in \NN$ choose $\delta_m > 0$ such that $|F(z) - F(x)| < 1/m$ whenever $|z-x| < \delta_m$. Choose $r_m, s_m \in \mathbb{Q}$ such that $r_m \leq z \leq s_m$ and $|r_m - z| < \delta_m$ and $|s_m - z| < \delta_m$. By the monotonicity of each $F_n$, and since $\omega \in \Omega'$ we have for all $m \in \NN$:
    \[F(r_m) = \lim_{n \to \infty}F_n(r_m;\omega) \leq \lim_{n \to \infty}F_n(z;\omega) \leq \lim_{n \to \infty}F_n(s_m;\omega) = F(s_m).\]
    Due to the choice of $r_m$ and $s_m$ we obtain:
    \[-\frac{1}{m} < F(r_m) - F(z) \leq \lim_{n \to \infty}F_n(z;\omega) - F(z) \leq F(s_m) - F(z) < \frac{1}{m}.\]
    The result now follows since $|\lim_{n \to \infty}F_n(z;\omega) - F(z)| < 1/m$ for all $m \in \NN$. 
\end{proof}

\begin{proof}[Proof of Lemma \ref{continuity_F_in_G}]
    Let $R > 0$ and assume $\lim_{n \to \infty} F_n(z) = F(z)$ for all $z \in [0,R)$. Fix $z \in [0,R)$. We introduce the following shorthand notation, for $h \in [0,z]$:
    \[\phi_n(h) := \int_0^h (h-t)^\frac{d}{2}\mathrm{d}F_n(t), \text{ and } \phi(h) := \int_0^h (h-t)^\frac{d}{2}\mathrm{d}F(t).\]
    The triangle inequality yields the following bound:
    \begin{align}
     \begin{split}
            \left| G_{F_n}(z) - G_F(z)\right| &\leq\left| \int_0^z \exp\left(-\kappa_d \int_0^h (h-t)^\frac{d}{2}\mathrm{d}F_n(t) \right) -  \exp\left(-\kappa_d \int_0^h (h-t)^\frac{d}{2}\mathrm{d}F(t) \right)\mathrm{d}F_n(h)\right| +\\
        &\phantom{\leq} + \left|\int_0^z \exp\left(-\kappa_d \int_0^h (h-t)^\frac{d}{2}\mathrm{d}F(t) \right)\mathrm{d}(F_n - F)(h) \right|
        \end{split}\nonumber \\
        \begin{split}
             &\leq \sup_{h \in [0,z]}\left|\exp\left(-\kappa_d \phi_n(h) \right) - \exp\left(-\kappa_d \phi_n(h) \right) \right|F_n(z) + \\
             &\phantom{\leq} + \left|\int_0^z \exp\left(-\kappa_d \phi(h) \right)\mathrm{d}(F_n - F)(h) \right|.\label{GF_continuity_lemma_proofeq}
        \end{split}
    \end{align}
    Let us consider the first term of (\ref{GF_continuity_lemma_proofeq}). Fix $h \in [0,z]$. Since $F_n$ converges pointwise to $F$ on $[0,z]$ and $t \mapsto (h-t)^\frac{d}{2}\mathds{1}\{h \geq t\}$ is continuous and bounded on $[0,z]$ it follows that $\lim_{n \to \infty}\phi_n(h) = \phi(h)$. Hence, the sequence of monotone functions $\exp\left(-\kappa_d \phi_n(\blank) \right)$ converges pointwise to the monotone function $\exp\left(-\kappa_d \phi(\blank) \right)$ on $[0,z]$. Because $\phi$ is (absolutely) continuous, the limit function $\exp\left(-\kappa_d \phi(\blank) \right)$ is continuous. The convergence is therefore uniform on $[0,z]$, and we obtain: 
    \[\lim_{n \to \infty}\sup_{h \in [0,z]}\left|\exp\left(-\kappa_d \phi_n(h) \right) - \exp\left(-\kappa_d \phi(h) \right) \right|F_n(z) = 0 \blank F(z) = 0. \]
    Let us now consider the second term of (\ref{GF_continuity_lemma_proofeq}). Because $\exp\left(-\kappa_d \phi(\blank) \right)$ is continuous and bounded, it immediately follows from the pointwise convergence of $F_n$ to $F$ on $[0,z]$ that this second term vanishes as $n \to \infty$. This proves that $\lim_{n \to \infty} G_{F_n}(z) = G_F(z)$. The proof remains valid when $R = \infty$.
\end{proof}

\begin{proof}[Proof of Theorem \ref{FV_identifiability_theorem}]
     Let $z \geq 0$. For $i \in \{1,2\}$ write $F_i^V := V(\blank;F_i,m_{F_i})$. From equation (\ref{FV_proof_eq2}) it can be seen that the (Lebesgue-Stieltjes) measures associated with $F_i^V$ and $F_i$ are mutually absolutely continuous. The corresponding Radon-Nikodym derivative is given by:
    \[\frac{\mathrm{d}F_i^V}{\mathrm{d}F_i}(z) = \pi\int_z^\infty \exp\left(-\pi \int_0^u F_i(t)\mathrm{d}t \right)\mathrm{d}u =: p_i(z).\]
    Hence, we may also write:
    \[F_i(z) = \int_0^z \frac{\mathrm{d}F_i}{\mathrm{d}F_i^V}(h)\mathrm{d}F_i^V(h) = \int_0^z \frac{1}{p_i(h)}\mathrm{d}F_i^V(h).\]
    Since $m_{F_i} < \infty$ we have $p_i(0) < \infty$ and from its definition it is clear that $p_i$ is a decreasing function. Because $x \mapsto 1/x$ is Lipschitz on $(c,\infty)$ for $c > 0$ with Lipschitz constant $1/c^2$ we have for $h \in [0,z]$:
    \[\left|\frac{1}{p_1(h)} - \frac{1}{p_2(h)} \right| \leq \max\left\{\frac{1}{p_1(h)^2},\frac{1}{p_2(h)^2}\right\}\left|p_1(h) - p_2(h)\right| \leq C(h) \left|p_1(h) - p_2(h)\right|. \]
    Here we have defined $C(h) := \max\{1/p_1(h)^2, 1/p_2(h)^2\}$, which is increasing. As a consequence, we obtain the following upper bound for $|F_1(z) - F_2(z)|$:
    \begin{align}
        |F_1(z) - F_2(z)| &= \left|\int_0^z \frac{1}{p_1(h)} - \frac{1}{p_2(h)}\mathrm{d}F_1^V(h) + \int_0^z\frac{1}{p_2(h)}\mathrm{d}(F_1^V - F_2^V)(h)\right| \nonumber\\
        &\leq C(z) \int_0^z \left|p_1(h) - p_2(h)\right|\mathrm{d}F_1^V(h) + \left| \int_0^z\frac{1}{p_2(h)}\mathrm{d}(F_1^V - F_2^V)(h)\right|. \label{FV_identifiable_proof_eq}
    \end{align}
    We consider the two terms in (\ref{FV_identifiable_proof_eq}) separately. The first term of (\ref{FV_identifiable_proof_eq}) is bounded by:
    \begin{align}
        \begin{split}
            &\phantom{\leq}\pi C(z)\int_0^z\left|\int_0^\infty \exp\left(-\pi \int_0^u F_1(t)\mathrm{d}t \right) - \exp\left(-\pi \int_0^u F_2(t)\mathrm{d}t \right)\mathrm{d}u \right|\mathrm{d}F_1^V(h) +\\
        &\phantom{\leq}+ \pi C(z)\int_0^z\left|\int_0^h \exp\left(-\pi \int_0^u F_1(t)\mathrm{d}t \right) - \exp\left(-\pi \int_0^u F_2(t)\mathrm{d}t \right)\mathrm{d}u \right|\mathrm{d}F_1^V(h). \label{FV_identifiable_proof_eq2}
        \end{split}
    \end{align}
    The first term of (\ref{FV_identifiable_proof_eq2}) is equal to $\pi C(z)F_1^V(z)\left|m_{F_1} - m_{F_2} \right|$ and the second term of (\ref{FV_identifiable_proof_eq2}) is bounded by:
    \begin{align}
        &\phantom{\leq} \pi C(z)\int_0^z\int_0^h\left|\exp\left(-\pi \int_0^u F_1(t)\mathrm{d}t \right) - \exp\left(-\pi \int_0^u F_2(t)\mathrm{d}t \right) \right| \mathrm{d}u\mathrm{d}F_1^V(h)\nonumber \\
        &\leq \pi^2 C(z)\int_0^z\int_0^h\int_0^u \left|F_1(t) - F_2(t) \right|\mathrm{d}t\mathrm{d}u\mathrm{d}F_1^V(h)\label{FV_identifiable_proof_eq3}\\
        &\leq \pi^2 C(z)F_1^V(z)z\int_0^z \left|F_1(t) - F_2(t) \right|\mathrm{d}t. \nonumber
    \end{align}
    In (\ref{FV_identifiable_proof_eq3}) we used the fact $|e^{-x}-e^{-y}| \leq |x-y|$ for $x,y \geq 0$. Via the integration by parts formula, the second term of (\ref{FV_identifiable_proof_eq}) is bounded by:
    \begin{align*}
        &\phantom{\leq}\left|\left(F_1^V(z) - F_2^V(z)\right) \frac{1}{p_2(z)} - \int_0^z F_1^V(h) - F_2^V(h)\mathrm{d}\left(\frac{1}{p_2(h)} \right)(h) \right|\\
        &\leq \frac{1}{p_2(z)}\left|F_1^V(z) - F_2^V(z)\right| + \sup_{h \in [0,z]}\left|F_1^V(h) - F_2^V(h) \right|\left|\int_0^z\mathrm{d}\left(\frac{1}{p_2(h)} \right)(h)\right| \\
        &\leq \frac{2}{p_2(z)}\sup_{h \in [0,z]}\left|F_1^V(h) - F_2^V(h) \right|.
    \end{align*}
    Collecting all results, we obtain:
    \begin{align*}
        |F_1(z) - F_2(z)| &\leq \pi C(z)F_1^V(z)\left|m_{F_1} - m_{F_2} \right| + \pi^2 C(z)F_1^V(z)z\int_0^z \left|F_1(t) - F_2(t) \right|\mathrm{d}t + \\
        &\phantom{\leq} + \frac{2}{p_2(z)}\sup_{h \in [0,z]}\left|F_1^V(h) - F_2^V(h) \right|.
    \end{align*}
    Applying Theorem \ref{gronwall_thm} and (\ref{gronwall_thm_note}) yields:
    \begin{align}
        |F_1(z) - F_2(z)| &\leq K(z)\left(\pi C(z)F_1^V(z)\left|m_{F_1} - m_{F_2} \right| + \frac{2}{p_2(z)}\sup_{h \in [0,z]}\left|F_1^V(h) - F_2^V(h) \right|\right).\label{FV_identifiable_proof_final_eq}
    \end{align}
    Here, $K(z)$ is given by:
    \[K(z):= \left(1 + \pi^2 C(z)F_1^V(z)z^2\exp\left(\pi^2 C(z)F_1^V(z)z^2 \right) \right)\]
    The statement of the theorem immediately follows from (\ref{FV_identifiable_proof_final_eq}).
\end{proof}

\begin{proof}[Proof of Lemma \ref{mf_continuity_lemma}]
    We first note that we may assume without loss of generality that $(F_n)_{n \geq 1}$ is a sequence of functions not containing the zero function. Indeed, we could take an arbitrary subsequence $(n_l)_{l \geq 1} \subset (n)_{n\geq 1}$, and then use the pointwise convergence of $F_n$ to $F$ to choose a further subsequence $(n_k)_{k\geq 1} \subset (n_l)_{l \geq 1}$ such that $(F_{n_k})_{k \geq 1}$ is a sequence which does not contain the zero function. If we then show $\lim_{k \to \infty}m_{F_{n_k}} = m_F$ then the whole sequence also converges: $\lim_{n \to \infty}m_{F_n} = m_F$.
    
    We introduce the following notation, for $u \geq 0$ let:
    \[p_n(u) := \exp\left(-\pi \int_0^u F_n(t)\mathrm{d}t\right), \text{\quad } p(u) := \exp\left(-\pi \int_0^u F(t)\mathrm{d}t\right).\]
    Via the inequality $|e^{-x}-e^{-y}| \leq |x-y|$ for $x,y \geq 0$ and (\ref{integration_by_parts}) we obtain the following upper bound for $|p_n(u) - p(u)|$:
    \begin{equation}
        |p_n(u) - p(u)| \leq \pi \left|\int_0^u F(t) - F_n(t)\mathrm{d}t \right| = \pi \left|\int_0^u (u-t)\mathrm{d}(F - F_n)(t) \right| . \label{pointwise_convergence_pn_eq}
    \end{equation}
    Due to the pointwise convergence of $F_n$ to $F$ we obtain that $p_n$ converges pointwise to $p$ as $n \to \infty$. The triangle inequality yields:
    \begin{equation}
        |m_{F_n} - m_F| \leq \int_0^z |p_n(u) - p(u)|\mathrm{d}u + \int_z^\infty |p_n(u) - p(u)|\mathrm{d}u. \label{mf_convergence_eq}
    \end{equation}
    The first term of (\ref{mf_convergence_eq}) vanishes as $n \to \infty$. Indeed, $p_n$ converges pointwise to $p$ as $n \to \infty$, and since $|p_n(u) - p(u)| \leq 1$ the dominated convergence theorem may be applied. The dominated convergence theorem can also be used to show that the second term of (\ref{mf_convergence_eq}) vanishes as $n \to \infty$. We now show which dominating function $g$ may be used. Choose $z \geq 0$ large enough such that $F(z) > 0$ and set $c := F(z)$. We show that $m_F < \infty$:
    \begin{align}
        m_F &= \int_0^z p(u)\mathrm{d}u + \int_z^\infty p(u)\mathrm{d}u\nonumber \\ 
        &\leq z + \exp\left(-\pi \int_0^z F(t)\mathrm{d}t\right)\int_z^\infty \exp\left(-\pi \int_z^u F(t)\mathrm{d}t\right)\mathrm{d}u\nonumber \\
        &\leq z + \int_z^\infty \exp\left(-\pi c \int_z^u \mathrm{d}t\right)\mathrm{d}u = z + \frac{1}{\pi c}. \label{mf_finite_bound}
    \end{align}
    Choose $N \in \NN$ large enough such that $F_n(z) \geq c/2$ for all $n \geq N$. This can be done since $F_n$ converges pointwise to $F$. Applying the same bound as in (\ref{mf_finite_bound}) yields $|p_n(u)| \leq \exp\left(-\pi c (u-z)/2 \right)$ for all $u \geq z$ and all $n \geq N$. Hence, we may define the dominating function $g:[z,\infty) \to [0,\infty)$ as:
    \[g(u):= p(u) + \max\left\{\max_{k \in \{1,\dots,N\}}p_k(u), \exp\left(-\pi \frac{c}{2} (u-z) \right)\right\}.\]
    Note that $m_F < \infty$ and $m_{F_k} < \infty$ for all $k \in \{1,\dots,N\}$ by (\ref{mf_finite_bound}), applied to $F$ and $F_k$ respectively. As a consequence, $g$ is integrable on $[z,\infty)$. Because $|p_n(u) - p(u)| \leq g(u)$ for all $u \geq z$ the proof is finished.
\end{proof}

\begin{proof}[Proof of Lemma \ref{FV_continuity_lemma}]
Let $z \geq 0$, we readily obtain the following bound:
    \begin{align}
        |V(z;F_n,&m_n) - V(z;F,m)| \leq \nonumber \\
        \begin{split}
        &\left|\exp\left(-\pi \int_0^z F(t)\mathrm{d}t\right) - \exp\left(-\pi \int_0^z F_n(t)\mathrm{d}t\right) \right| +\\
        &+ \pi \left| F(z) - F_n(z)\right| \left|m - \int_0^z  \exp\left(-\pi \int_0^{u} F(t)\mathrm{d}t \right)\mathrm{d}u\right| +\\
        &+ \pi F_n(z)\left(|m_n - m| + \left|\int_0^z  \exp\left(-\pi \int_0^{u} F_n(t)\mathrm{d}t \right)- \exp\left(-\pi \int_0^{u} F(t)\mathrm{d}t \right)\mathrm{d}u\right| \right)
        \end{split}\label{FV_continuity_lemma_proof}
    \end{align}
    Each of the three terms of (\ref{FV_continuity_lemma_proof}) vanishes as $n \to \infty$, and each of the terms appearing here also appear in the proof of Lemma 5.5. The fact that the first term vanishes follows from (\ref{pointwise_convergence_pn_eq}). The second term vanishes due to the pointwise convergence of $F_n$ to $F$. The third term vanishes since $\lim_{n \to \infty}F_n(z) = F(z)$, $\lim_{n \to \infty} m_n = m$, and by using the same argument as for the first term in (\ref{mf_convergence_eq}).
\end{proof}

\section{Computational formula}\label{appendix_mn}
First of all, note that:
\begin{equation}
    \hat{m}_n = \int_0^{h_1}\exp\left(-\pi\int_0^u \hat{F}_n^0(t)\mathrm{d}t\right)\mathrm{d}u + \int_{h_1}^\infty \exp\left(-\pi\int_0^u  \hat{F}_n^0(t)\mathrm{d}t\right)\mathrm{d}u \label{mn_integral_split}
\end{equation}
The first integral of (\ref{mn_integral_split}) is equal to $h_1$, since $\hat{F}_n^0$ is zero on $[0,h_1)$. Let $h_{k+1} > h_k$, via a direct computation we obtain:
\begin{align*}
    \int_{h_1}^{h_{k+1}}&\exp\left(-\pi\int_0^u  \hat{F}_n^0(t)\mathrm{d}t\right)\mathrm{d}u =\\ 
    &=\sum_{i=2}^{k+1} \int_{h_{i-1}}^{h_i}\exp\left(-\pi\int_0^u  \hat{F}_n^0(t)\mathrm{d}t\right)\mathrm{d}u \\
    &= \sum_{i=2}^{k+1}\exp\left(-\pi\int_0^{h_{i-1}} \hat{F}_n^0(t)\mathrm{d}t\right) \int_{h_{i-1}}^{h_i}\exp\left(-\pi\int_0^u  \hat{F}_n^0(t)\mathrm{d}t\right)\mathrm{d}u \\
    &= \sum_{i=2}^{k+1}\exp\left(-\pi\int_0^{h_{i-1}} \hat{F}_n^0(t)\mathrm{d}t\right) \int_{h_{i-1}}^{h_i}\exp\left(-\pi(u - h_{i-1})\hat{F}_n^0(h_{i-1})\right)\mathrm{d}u \\
    &= \sum_{i=2}^{k+1}\exp\left(-\pi\sum_{j=1}^{i-1}\hat{F}_n^0(h_j)(h_j - h_{j-1}) \right)\frac{1}{\pi \hat{F}_n^0(h_{i-1})}\left(1-\exp\left(-\pi \hat{F}_n^0(h_{i-1})(h_i - h_{i-1}) \right) \right). 
\end{align*}
Letting $h_{k+1} \to \infty$ we obtain:
\begin{align*}
    \hat{m_n} &= h_1 + \exp\left(-\pi\sum_{j=1}^{k}\hat{F}_n^0(h_j)(h_j - h_{j-1}) \right)\frac{1}{\pi \hat{F}_n^0(h_{k})} + \\
    &\phantom{=} + \sum_{i=2}^{k}\exp\left(-\pi\sum_{j=1}^{i-1}\hat{F}_n^0(h_j)(h_j - h_{j-1}) \right)\frac{1}{\pi \hat{F}_n^0(h_{i-1})}\left(1-\exp\left(-\pi \hat{F}_n^0(h_{i-1})(h_i - h_{i-1}) \right) \right).
\end{align*}

\bibliographystyle{abbrv}
\addcontentsline{toc}{section}{References}
\bibliography{export}

\end{document}